\documentclass[onefignum,onetabnum]{siamart190516}

\usepackage{amssymb}
\usepackage{xcolor}
\usepackage{graphicx}
\usepackage{booktabs}
\usepackage{multirow}

\newcommand{\NN}{\mathbb{N}}
\newcommand{\RR}{\mathbb{R}}

\newcommand{\ba}{\boldsymbol{a}}

\newcommand{\be}{\boldsymbol{e}}
\newcommand{\bn}{\boldsymbol{n}}
\newcommand{\br}{\boldsymbol{r}}
\newcommand{\bu}{\boldsymbol{u}}
\newcommand{\bv}{\boldsymbol{v}}
\newcommand{\bw}{\boldsymbol{w}}
\newcommand{\bx}{\boldsymbol{x}}
\newcommand{\bW}{\boldsymbol{W}}

\newcommand{\blambda}{\boldsymbol{\lambda}}
\newcommand{\bff}{\boldsymbol{f}}
\newcommand{\bsigma}{\boldsymbol{\sigma}}

\newcommand{\bH}{\boldsymbol{H}}

\newcommand{\bT}{\boldsymbol{\mathrm{T}}}
\newcommand{\bI}{\boldsymbol{\mathrm{I}}}

\newcommand{\bD}{\boldsymbol{\mathrm{D}}}
\newcommand{\bE}{\mathrm{E}}
\newcommand{\bF}{\boldsymbol{\mathrm{F}}}
\newcommand{\dd}{\mathrm{d}}

\newcommand{\bbP}{\mathbb{P}}
\newcommand{\bbQ}{\mathbb{Q}}
\newcommand{\bbI}{\mathbb{I}}

\makeatletter
\newcommand{\opnorm}{\@ifstar\@opnorms\@opnorm}
\newcommand{\@opnorms}[1]{%
  \left|\mkern-1.5mu\left|\mkern-1.5mu\left|
   #1
  \right|\mkern-1.5mu\right|\mkern-1.5mu\right|
}
\newcommand{\@opnorm}[2][]{%
  \mathopen{#1|\mkern-1.5mu#1|\mkern-1.5mu#1|}
  #2
  \mathclose{#1|\mkern-1.5mu#1|\mkern-1.5mu#1|}
}
\makeatother

\newcommand{\nrm}[1]{\left\lVert#1\right\rVert}

\newcommand{\Vo}{\mathring{V}}
\newcommand{\Ko}{\mathring{K}}

\newcommand{\Cc}{\mathcal{C}}
\newcommand{\bCc}{\boldsymbol{\mathcal{C}}}
\newcommand{\Ec}{\mathcal{E}}
\newcommand{\Jc}{\mathcal{J}}

\newcommand{\bPc}{\boldsymbol{\mathcal{P}}}
\newcommand{\Pc}{\mathcal{P}}
\newcommand{\Tc}{\mathcal{T}}

\DeclareMathOperator*{\Ran}{Ran}
\DeclareMathOperator*{\Ker}{Ker}
\DeclareMathOperator*{\Span}{span}

\newtheorem{remark}[theorem]{Remark}

\begin{document}

\title{On the finite element approximation of a semicoercive Stokes variational inequality arising in glaciology}

\author{Gonzalo G.~de Diego\thanks{Mathematical Institute, University of Oxford, Oxford, OX2 6GG, UK (gonzalezdedi@maths.ox.ac.uk). PEF was supported by EPSRC grants EP/V001493/1 and EP/R029423/1.} \and Patrick E.~Farrell\footnotemark[1] \and Ian J.~Hewitt\footnotemark[1]}

\maketitle

\begin{abstract}
	Stokes variational inequalities arise in the formulation of glaciological problems involving contact. We consider the problem of a two-dimensional marine ice sheet with a grounding line, although the analysis presented here is extendable to other contact problems in glaciology, such as that of subglacial cavitation. The analysis of this problem and its discretisation is complicated by the nonlinear rheology commonly used for modelling ice, the enforcement of a friction boundary condition given by a power law, and the presence of rigid modes in the velocity space, which render the variational inequality semicoercive. In this work, we consider a mixed formulation of this variational inequality involving a Lagrange multiplier and provide an analysis of its finite element approximation. Error estimates in the presence of rigid modes are obtained by means of a specially-built projection operator onto the subspace of rigid modes and a Korn-type inequality. These proofs rely on the fact that the subspace of rigid modes is at most one-dimensional. Numerical results are reported to validate the error estimates.
\end{abstract}



\section{Introduction}\label{sec:intro}

We consider the problem of a marine ice sheet resting on a bedrock and sliding into the ocean, where it goes afloat. Such a configuration is found in Greenland and Antarctica, and the dynamics of the grounding line, the point where ice loses contact with the bedrock, is of crucial importance for predicting future sea level rise and comprehending large scale climate dynamics \cite{weertman1974,schoof2007b,ritz2015,garbe2020}. This contact problem is modelled by coupling a Stokes problem for the ice flow with a time-dependent advection equation for the free surface. At each instant in time, the Stokes equation must be solved with contact boundary conditions that allow the detachment of the ice from the bedrock. These contact conditions transform the instantaneous Stokes problem into a variational inequality. Similar contact conditions appear in a related problem of subglacial cavitation (see Remark \ref{remark:subglacial_cavitation} below), which is also of fundamental importance in glaciology \cite{fowler1986, schoof2005, zoet2016}, and the results presented in this work are extendable to such problems \cite{dediego2022}.

Numerous finite element simulations of these equations have been carried out \cite{durand2009, favier2012, stubblefield2021}. However, to the best of our knowledge, no formal analysis of these problems and their approximation exist in the mathematical literature. Moreover, we believe that the discretisations used in these computations can be improved upon, by exploiting the structure of the variational inequality. Although the Stokes variational inequality is superficially similar to the elastic contact problem, which has been widely studied \cite{kikuchi1988,haslinger1996}, the Stokes problem in the context of marine ice sheets with a grounding line includes three substantial difficulties that must be addressed carefully: the presence of rigid body modes in the space of admissible velocities, the nonlinear rheological law used to model ice as a viscous fluid, and the nonlinearity of the friction boundary condition.

In this work we analyse the instantaneous Stokes variational inequality and its approximation. The presence of rigid body modes renders this problem semicoercive. Although semicoercive variational inequalities have been studied in the past \cite{gwinner1991, spann1994, adly2000}, existing analyses use purely indirect arguments which give very limited information on how different meshes and finite elements affect the discretisation. Here, we present a novel constructive approach based on the use of a specially designed projection operator onto the subspace of rigid modes that satisfies a Korn type inequality. Error estimates are obtained for the rigid component of the velocity error by exploiting the fact that the dimension of the subspace of rigid modes is at most one. The nonlinear rheology and friction boundary condition complicate the estimation of errors for the discrete problem. Here, we use the techniques from \cite{belenki2012,hirn2013} to establish a convergence analysis. 

We propose a mixed formulation of the Stokes variational inequality where a Lagrange multiplier is used to enforce the contact conditions. This formulation permits a structure-preserving discretisation that explicitly enforces a discrete version of the contact conditions, up to rounding errors. This allows for a precise distinction between regions where the ice detaches from the bed and those where it remains attached. This precision is extremely useful when coupling the Stokes variational inequality with the time-dependent advection equation for the free surface. Numerical results with this scheme in the context of subglacial cavitation can be found in \cite{dediego2022}.

\subsection{Outline of the paper}

In Section \ref{sec:formulation}, the Stokes variational inequality and its mixed formulation are presented. We prove a Korn-type inequality involving a projection operator onto the subspace of rigid modes that will be used throughout the analysis, and we demonstrate that the mixed formulation is well posed. In Section \ref{sec:abstract-discretisation}, we analyse a family of finite element approximations of the mixed problem and present error estimates in terms of best approximation results for the velocity, pressure and Lagrange multiplier. Finally, in Section \ref{sec:fe-scheme}, a concrete finite element scheme involving quadratic elements for the velocity and piecewise constant elements for the pressure and the Lagrange multiplier is introduced. We then present error estimates for this scheme and we solve a problem with a manufactured solution to calculate convergence rates and compare these with our estimates.

\subsection{Notation}

Given two normed vector spaces $X$ and $Y$ and a bounded linear operator $T:X\to Y$, the dual of $X$ is denoted by $X'$ and the dual operator to $T$ by $T':Y'\to X'$. The range of $T$ is denoted by $\Ran{T}$ and its kernel by $\Ker{T}$. The norm in $X$ is denoted by $\nrm{\cdot}_X$ and the pairing between elements in the primal and dual spaces by $\langle f, x \rangle_X$ for $f\in X'$ and $x\in X$. We will work with the Lebesgue and Sobolev spaces $W^{m,r}(\Omega)$, where $m \geq 0$ and $r \geq 1$, defined as the set of functions with weak derivatives up to order $m$ which are $r$-integrable. When $m = 0$ we write $L^r(\Omega)$. The space of polynomials of degree $k$ over a simplex $E$ (interval, triangle, tetrahedron) is denoted by $\Pc_k(E)$. The space of continuous functions over a domain $\Omega$ is given by $\Cc(\Omega)$. Vector-valued functions and vector-valued function spaces will be denoted with bold symbols, e.g.~$\bu$ and $\boldsymbol{W}^{m,r}(\Omega)$. We write $f\sim g$, $f \lesssim g$ and $f \gtrsim g$ if there exist generic constants $c,C > 0$ such that $cf \leq g \leq Cf$, $c f \leq g$, and $cf \geq g$, respectively. Throughout this work, we assume that these generic constants do not depend on the mesh size or on the continuous and discrete solutions of the problem.

\section{Formulation of the problem}\label{sec:formulation}

In this section we introduce the semicoercive variational inequality that arises in the study of marine ice sheets and present its formulation as a mixed problem with a Lagrange multiplier. We then analyse the existence and uniqueness of solutions for the mixed problem.

\begin{figure}
	\centering
	\includegraphics[scale = 0.75]{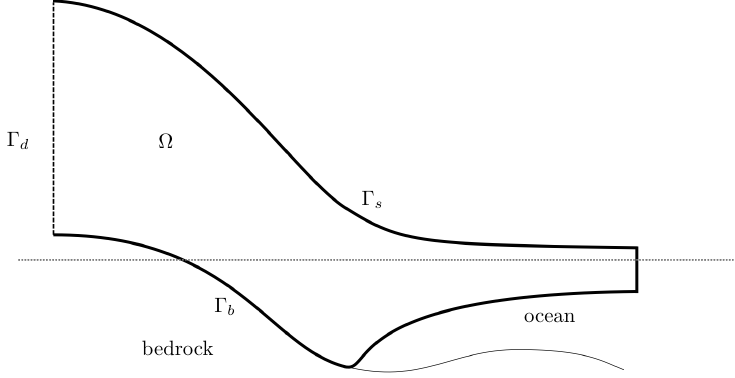}
	\caption{Geometry of the problem under consideration. The domain $\Omega$ represents a half of a symmetric marine ice sheet. The boundary of $\Omega$ is partitioned into $\Gamma_b$ (the ice-bedrock interface), $\Gamma_s$ (the ice-ocean and ice-atmosphere interface), and $\Gamma_d$ (the symmetry axis). The horizontal dotted line represents the sea level.}
	\label{fig:domain}
\end{figure}

\subsection{A model for ice flow}

We consider a two-dimensional symmetrical marine ice sheet resting on a bedrock and sliding into the ocean. This is the most common configuration considered when studying marine ice sheets \cite{schoof2007a, schoof2007b, durand2009} and is generally used as a benchmark test case \cite{pattyn2012}. We denote by $\Omega\subset\RR^2$ the domain which represents one half of the ice sheet and we assume it to be connected and polygonal. The latter assumption is made to simplify the analysis, but we expect the essential results presented here to extend to domains with smooth enough boundaries. Ice is generally modelled as a viscous incompressible flow whose motion is described by the Stokes equation \cite{fowler2011}:
\begin{subequations}\label{eq:main}
	\begin{align}
		- \nabla \cdot \left( 2 \eta(|\bD\bu|) \bD\bu \right) + \nabla p &= \bff && \text{in $\Omega$}, \label{eq:main-subeq1}\\
		\nabla \cdot \bu &= 0 && \text{in $\Omega$}.\label{eq:main-subeq2}
	\end{align}
\end{subequations}	
In the equations above, $\bu: \Omega \to \mathbb{R}^2$ represents the ice velocity, $p: \Omega \to \mathbb{R}$ the pressure and $\bff: \Omega \to \mathbb{R}^2$ is a prescribed body force, generally due to gravitational forces. The tensor $\bD\bu$ is the symmetric part of the velocity gradient, that is,
\[
	\bD\bu = \frac{1}{2}\left( \nabla\bu + \nabla\bu^\top \right).
\]
The coefficient $\eta(|\bD\bu|)$ is the effective viscosity of ice, which relates the stress and strain rates. A power law, usually called Glen's law \cite{glen1958}, is the most common choice of rheological law for ice:
\begin{align}\label{eq:glens_law}
	\eta(|\bD\bu|) = \frac{1}{2} \mathcal{A}^{-1/n} \left( \frac{1}{2} |\bD\bu|^2 \right)^\frac{1-n}{2n}.
\end{align}
Here, $|\cdot|$ represents the Frobenius norm of a matrix: for $B\in\RR^{m\times m}$ with components $B_{ij}$ we have $|B|^2 = \sum_{ij}B_{ij}$. The field $\mathcal{A}\in L^\infty(\Omega)$ is a prescribed function for which $\mathrm{ess\, inf} \mathcal{A} > 0$. The parameter $n$ is constant and is usually set to $n=3$; for $n=1$ we recover the standard linear Stokes flow. From now on, we simply write 
\begin{align}\label{eq:rheology}
	\eta(|\bD\bu|) = \frac{1}{2}\alpha |\bD\bu|^{r-2},
\end{align}
where $\alpha = (1/2)^{(r-2)/2}\mathcal{A}^{1-r}$ is in $L^\infty(\Omega)$ and satisfies $\alpha \geq \alpha_0$ a.e.~on $\Omega$ for some $\alpha_0 > 0$. Moreover, $r = 1 + 1/n$ is in $(1,2]$ for $n \geq 1$. This expression for $\eta$ reveals the $r$-Stokes nature of the problem when considered as a variational problem in the setting of Sobolev spaces.

\subsection{Boundary conditions}

For a given velocity and pressure field, we define the stress tensor $\sigma = \sigma(\bu,p)$ by 
\[
	\sigma = \alpha |\bD\bu|^{r-2}\bD\bu - pI,
\]
where $I:\Omega\to\RR^{2\times 2}$ is the identity tensor field. Let $\bn$ denote the unit outward-pointing normal vector to the boundary $\partial\Omega$ and $\bT = \bI - \bn \bn^\top$ the orthogonal projection onto the tangential component to the boundary. We define the normal and tangential stresses at the boundary as
\begin{align*}
	\sigma_{nn} = (\sigma\bn)\cdot \bn \quad \text{and} \quad \bsigma_{nt} = \bT\sigma\bn.
\end{align*}

The boundary $\partial\Omega$ is partitioned into three disjoint open sets $\Gamma_s$, $\Gamma_b$ and $\Gamma_d$ of positive measure, see Figure \ref{fig:domain}. The subset $\Gamma_s$ represents the part of the boundary in contact with the atmosphere and the ocean. Here we enforce
\begin{align}\label{eq:bc-gammas}
	\sigma_{nn} = p_s \quad \text{and} \quad \bsigma_{nt} = 0 \quad & \text{on $\Gamma_s$},
\end{align}
where $p_s:\Gamma_s\to\RR$ represents a prescribed surface traction force. On $\Gamma_b$ the ice is in contact with the bedrock. Here, we enforce the contact conditions which allow the ice to detach from but not penetrate the bedrock. In particular, detachment can occur if the normal stress equals the subglacial water pressure, which is defined everywhere along a thin lubrication layer in between the ice and the bedrock. We also assume that the ice slides along the bedrock according to a power law. Then, the boundary conditions on $\Gamma_b$ are given by 
\begin{subequations}\label{eq:bc-gammab}
	\begin{align}
		\bu\cdot\bn \leq 0,\quad \sigma_{nn} \leq -p_w \quad \text{and} \quad \left(\bu\cdot\bn\right)\left(\sigma_{nn} + p_w \right) = 0 & \quad \text{on $\Gamma_b$}, \label{eq:main-bc-contact} \\
		\bsigma_{nt} = - \tau | \bT\bu |^{r-2}\bT\bu & \quad \text{on $\Gamma_b$},\label{eq:main-bc-sliding}
	\end{align}
\end{subequations}
where $p_w : \Gamma_b \to \RR$ is the water pressure at the ice-bedrock interface and $\tau > 0$ a constant. A power-law boundary condition as in \eqref{eq:main-bc-sliding} was first proposed by Weertman \cite{weertman1957} and has since become a popular model for glacier sliding \cite{greve2009, fowler2011}.

Finally, $\Gamma_d$ represents the ice divide of the ice sheet, which is essentially its symmetry axis. As such, it is a vertical surface on which we enforce the symmetry conditions
\begin{align}\label{eq:bc-gammad}
	\bu\cdot\bn = 0 \quad \text{and} \quad \bsigma_{nt} = 0 \quad & \text{on $\Gamma_d$}.
\end{align}

\subsection{The mixed formulation}

We now present the mixed formulation whose analysis and approximation is the focus of this work. To do so, we first write \eqref{eq:main} with boundary conditions \eqref{eq:bc-gammas}-\eqref{eq:bc-gammad} as a variational inequality. Then, we introduce the mixed formulation by defining a Lagrange multiplier which enforces a constraint that arises due to the contact boundary conditions \eqref{eq:main-bc-contact}. In Appendix \ref{app:equivalence-forms} we specify and prove the sense in which these different formulations are equivalent.

In order to build a weak formulation of \eqref{eq:main} and \eqref{eq:bc-gammas}-\eqref{eq:bc-gammad}, we must first define suitable function spaces in which to seek the velocity and the pressure. For $r' = 1/(r-1)$, we write 
\begin{align*}
	V = \left\lbrace \bv \in \bW^{1,r}(\Omega) : \bv\cdot\bn = 0 \quad \text{on $\Gamma_d$} \right\rbrace , \quad Q = L^{r'}(\Omega).
\end{align*}
We denote by $\gamma_n : V \to L^r(\Gamma_b)$ the normal trace operator onto $\Gamma_b$. This operator is built by extending to $V$ the operator $\bv \mapsto \bv\cdot\bn$ on $\Gamma_b$, defined on smooth functions. The closed convex subset $K$ of $V$ is then defined by
\begin{align*}
	K = \left\lbrace \bv\in V : \gamma_n\bv \leq 0 \quad \text{a.e. on $\Gamma_b$} \right\rbrace.
\end{align*}
We also introduce the operators $A:V\to V'$, $G:V \to V'$ and $B:Q\to V'$, defined by
\begin{align}\label{eq:defn-operator-AB}
	\langle A\bu,\bv \rangle_V &= \int_\Omega \alpha |\bD\bu|^{r-2} \left(\bD\bu : \bD\bv\right)\,\dd x, \\
	\langle G\bu,\bv \rangle_V &= \int_{\Gamma_b} \tau |\bT\bu|^{r-2} \left(\bT\bu \cdot \bT\bv\right)\,\dd s, \\
	 \langle Bq, \bv \rangle &= \int_\Omega \left(\nabla\cdot\bv\right) q\,\dd x.
\end{align}
Moreover, the action of the applied body and surface forces on the domain $\Omega$ is expressed via the function $F\in V'$, defined as
\begin{align}\label{eq:defn-operator-F}
	\langle F, \bv \rangle_V = \int_\Omega \bff\cdot \bv\,\dd x + \int_{\Gamma_s}p_s\left( \bv\cdot\bn\right)\,\dd s - \int_{\Gamma_b}p_w\left( \bv\cdot\bn\right)\,\dd s.
\end{align}
In order for \eqref{eq:defn-operator-F} to make sense, we require $\bff\in L^{r'}(\Omega)$, $p_s \in L^{r'}(\Gamma_s)$ and $p_w \in L^{r'}(\Gamma_b)$. Then, \eqref{eq:main} with boundary conditions \eqref{eq:bc-gammas}-\eqref{eq:bc-gammad} can be reformulated as the variational inequality: find $(\bu,p)\in K\times Q$ such that
\begin{align}\label{eq:vi-continuous}
	\langle A\bu + G\bu - Bp - F, \bv - \bu \rangle_V + \langle Bq, \bu \rangle_V \geq 0 \quad \forall (\bv,q)\in K\times Q.
\end{align}

In the mixed formulation, the constraint $\bv\cdot\bn \leq 0$ on $\Gamma_b$ is enforced via a Lagrange multiplier. We denote the range of $\gamma_n$ by $\Sigma$ and equip this space with the $W^{1-1/r,r}(\Gamma_b)$ norm. We assume the geometry of $\Omega$ and $\Gamma_b$ to be sufficiently regular for this space to be a Banach space, see \cite[Section 5]{kikuchi1988}, \cite[Chapter III]{haslinger1996} and \cite[Chapter 7]{adams2003} for discussions on normal traces and trace spaces. The Lagrange multiplier is sought in the convex cone of multipliers
\[
	\Lambda = \left\lbrace \mu\in \Sigma' : \langle \mu, \zeta \rangle_\Sigma \geq 0 \quad \text{$\forall \zeta\in\Sigma$ s.t. $\zeta \leq 0$} \right\rbrace.
\]
The equivalent mixed formulation of \eqref{eq:vi-continuous} is: find $(\bu,p,\lambda)\in V\times Q\times \Lambda$ such that
\begin{subequations}\label{eq:mixed-continuous}
\begin{align}
	\langle A\bu + G \bu - Bp - F, \bv \rangle_V - \langle \lambda, \gamma_n\bv\rangle_\Sigma &=0 && \forall\bv\in V, \label{eq:mixed-continuous-V}\\
	 \langle Bq,\bu \rangle_V &= 0 && \forall q\in Q,\\
	 \langle \mu - \lambda, \gamma_n\bu\rangle_\Sigma &\geq 0 && \forall \mu\in \Lambda. \label{eq:mixed-continuous-S}
\end{align}
\end{subequations}

The Lagrange multiplier $\lambda$ essentially coincides with $\sigma_{nn} + p_w$ on $\Gamma_b$. Indeed, if the solution to \eqref{eq:mixed-continuous} is sufficiently smooth for integration by parts to hold, we arrive at $\lambda = \sigma_{nn} + p_w$ on $\Gamma_b$. Moreover, the conditions $\lambda\in\Lambda$ and \eqref{eq:mixed-continuous-S} are equivalent to
\begin{align}\label{eq:weak-contact-conditions}
	\langle \mu, \gamma_n\bu\rangle_\Sigma \geq 0 \quad \forall \mu\in\Lambda, \quad \lambda\in\Lambda \quad \text{and} \quad \langle \lambda, \gamma_n\bu \rangle_\Sigma = 0,
\end{align}
which is a weak representation of the contact boundary conditions \eqref{eq:main-bc-contact}.

\subsection{Well posedness of the mixed formulation}

Questions on the existence and uniqueness of solutions of the mixed system \eqref{eq:mixed-continuous} can be answered by studying an equivalent minimisation problem. This equivalence depends on the so-called inf-sup property holding for the operators $B$ and $\gamma_n$. Let 
\begin{align}
	V_b = \left\lbrace \bv\in V : \gamma_n \bv = 0 \quad \text{a.e.~on $\Gamma_b$}  \right\rbrace.
\end{align}
These inf-sup conditions can be stated as
\begin{align}
	\sup_{\bv\in V_b} \frac{\langle Bq, \bv \rangle_V}{\nrm{\bv}_V} &\gtrsim \nrm{q}_Q \quad \forall q\in Q, \label{eq:inf-sup-VbQ} \\
	\sup_{\bv\in V} \frac{\langle \mu, \gamma_n\bv\rangle_\Sigma}{\nrm{\bv}_V} &\gtrsim \nrm{\mu}_{\Sigma'} \quad  \forall \mu\in \Sigma'. \label{eq:inf-sup-VS}
\end{align}
Condition \eqref{eq:inf-sup-VbQ} is proved in \cite[Lemma 3.2.7]{jouvet2010} and \eqref{eq:inf-sup-VS} follows from the inverse mapping theorem because $\gamma_n$ is surjective onto the Banach space $\Sigma$. We also define the space of divergence-free functions $\Vo$ and the convex set $\Ko$ as
\begin{align*}
	\Vo = \left\lbrace \bv\in V : \nabla \cdot \bv = 0 \right\rbrace \quad \text{and} \quad \Ko = \Vo \cap K.
\end{align*}
Then, \eqref{eq:mixed-continuous} is equivalent to the minimisation of the functional
\begin{align}\label{eq:Jc-functional}
	\Jc(\bv) = \frac{1}{r}\int_\Omega \alpha | \bD\bv |^r\,\dd x + \frac{1}{r}\int_{\Gamma_b} \tau | \bT\bv |^r\,\dd s - \langle F, \bv \rangle_V
\end{align}
over $\Ko$, see Appendix \ref{app:equivalence-forms}. The existence of minimisers of $\Jc$ over $\Ko$ hinges on whether the set
\begin{align*}
	R_V = \left\lbrace \br\in V : \int_\Omega | \bD\br |^r\,\dd x + \int_{\Gamma_b} | \bT\br |^r\,\dd s = 0 \right\rbrace
\end{align*}
is equal to or larger than the trivial set $\{0\}$. As shown in \cite[Lemma 6.1]{kikuchi1988}, the kernel of $\bD$ coincides with the set of rigid modes in $\Omega$, defined by
\begin{align*}
	R = \left\lbrace \br\in\bH^1(\Omega):\br(x,y) = \left(\begin{array}{c}
	a \\
	b
\end{array}	\right) + \omega \left(\begin{array}{r}
	-y \\
	x
\end{array}	\right), \quad (a,b,\omega)\in\RR^3\right\rbrace.
\end{align*}
Hence, $R_V$ is the set of rigid modes $\br\in R$ satisfying $\bT\br = 0$ on $\Gamma_b$ and $\br\cdot\bn = 0$ on $\Gamma_d$. For this reason, the dimension of $R_V$ can be at most 1 whenever $\Gamma_b$ is a flat surface perpendicular to $\Gamma_d$. In this case, $R_V$ is given by purely vertical translations. 

\begin{remark}\label{remark:subglacial_cavitation}
	Although a flat bedrock may appear to be unrealistic, these are considered in many theoretical studies of marine ice sheets \cite{schoof2007a,pattyn2012,stubblefield2021}. One-dimensional subspaces of rigid modes in $V$ also arise in marine ice sheets which can slide freely ($\tau = 0$) and in the subglacial cavity problem considered in \cite{gagliardini2007, dediego2022} (whenever a horizontal velocity boundary condition is imposed at the top boundary). An analysis of these two problems and their discretisation can be completed using the techniques and steps presented in this paper.
\end{remark}

We define the projection operator $\bbP:V\to R_V$ by
\begin{align*}
	\bbP(\bv) = \begin{cases}
		\dfrac{\int_{\Gamma_b} \bv\cdot\bn\,\dd s}{\int_{\Gamma_b} \be_R\cdot\bn\,\dd s} \be_R & \text{if $\dim{R_V} = 1$,} \\
		0 & \text{if $\dim{R_V} = 0$,}
	\end{cases}
\end{align*}
where $\be_R\in R_V$ is a basis function that spans $R_V$ when $\dim{R_V} = 1$. We choose this projection operator because it satisfies $\bbP(K) \subset K$. The operator $\bbQ = \bbI - \bbP$ then maps elements in $V$ onto a closed subspace whose intersection with $R_V$ is $\{0\}$. As a result, we have the following variation of Korn's inequality:

\begin{lemma}
	The inequality
	\begin{align}\label{eq:korn-inequality-variation}
		\nrm{\bbQ\bv}_V \lesssim \nrm{\bD\bv}_{L^r(\Omega)} + \nrm{\bT\bv}_{L^r(\Gamma_b)}
	\end{align}
	holds uniformly for all $\bv\in V$.
\end{lemma}

\begin{proof}
	Following the proof of \cite[Lemma 3]{chen2013}, we first notice that \eqref{eq:korn-inequality-variation} follows from 
	\begin{align}\label{eq:proof-korn-1}
		\int_\Omega |\bv|^r\,\dd x \lesssim \int_\Omega |\bD\bv|^r\,\dd x + \int_{\Gamma_b} |\bT\bv|^r\,\dd s \quad \forall \bv \in \Ran{\bbQ}
	\end{align}
	due to the generalised Korn inequality \cite[Lemma 2]{chen2013}. Since $\Ran{\bbQ}\cap R_V = \{0\}$, the proof of \eqref{eq:korn-inequality-variation} is completed by assuming \eqref{eq:proof-korn-1} to be false and mimicking the steps in the proof of \cite[Lemma 3]{chen2013}.
\end{proof}

Whenever $R_V \neq \{0\}$, \eqref{eq:mixed-continuous} is semicoercive in the sense that the operator $A+G$ has a nontrivial kernel. In Theorem \ref{thm:continuous-well-posed} below, we show that a consequence of semicoercivity is that \eqref{eq:mixed-continuous} will have a solution only when the following compatibility condition holds:
\begin{align}\label{eq:compatibility-condition}
	\langle F, \br \rangle_V < 0 \quad \forall \br\in (R_V \cap K) \setminus \{ 0 \}.
\end{align}
Condition \eqref{eq:compatibility-condition} allows us to establish the well-posedness of \eqref{eq:mixed-continuous} and the error estimates, because the restriction of the map $\br \mapsto \langle F, \br \rangle_V$ to the boundary of the unit ball in $K\cap R_V$ is a continuous map defined over a compact set. Therefore, whenever \eqref{eq:compatibility-condition} holds, we have the inequality
\begin{align}\label{eq:inequality-F}
	\delta\nrm{\br}_V \leq - \langle F, \br \rangle_V \quad \forall \br\in R_V\cap K,
\end{align}
where 
\begin{align*}
	\delta = \min_{\substack{\br\in R_V\cap K, \\ \nrm{\br}_V = 1}}{- \langle F, \br \rangle_V}.
\end{align*}
Inequality \eqref{eq:inequality-F} is used to prove that the solutions to the continuous and discrete problems are bounded from above in Theorems \ref{thm:continuous-well-posed} and \ref{thm:discrete-well-posed} below, respectively. It is also used in the proof of Lemma \ref{lemma:error-estimates-rigid-error} to obtain error estimates for the rigid component of the velocity error.

The importance of the compatibility condition \eqref{eq:compatibility-condition} is well-known in the study of semicoercive variational inequalities, see \cite{hlavacek1978, spann1994, kikuchi1988} in the context of general variational inequalities and \cite{schoof2010,chen2013} in a glaciological setting. The compatibility condition has the geometrical interpretation that the applied force $F$ should have an obtuse angle with the directions of escape of the body given by $R_V \cap K$, which in this case correspond with vertical upward movements whenever $\Gamma_b$ is flat.

\begin{theorem}\label{thm:continuous-well-posed}
	If $R_V = \{0\}$, then a solution to \eqref{eq:mixed-continuous} exists and is unique. If $R_V \neq \{0\}$, then there is a unique solution to \eqref{eq:mixed-continuous} provided the compatibility condition \eqref{eq:compatibility-condition} holds. Conversely, if $R_V \neq \{0\}$ and a solution exists, we have that
	\begin{align}\label{eq:compatibility-condition-weak}
		\langle F, \br \rangle_V \leq 0 \quad \forall \br\in R_V \cap K.
	\end{align}		
	 Moreover, a solution $(\bu,p,\lambda)\in V\times Q\times \Lambda$ of \eqref{eq:mixed-continuous} is bounded from above, i.e.
	\begin{align}\label{eq:bound-solution}
		\nrm{\bu}_V + \nrm{p}_Q + \nrm{\lambda}_{\Sigma'} \lesssim 1,
	\end{align}
	if \eqref{eq:compatibility-condition} holds when $R_V \neq \{0\}$.
\end{theorem}

\begin{proof}
	If $R_V = \{0\}$, then $\bbQ = \bbI$ and we can establish the coercivity of $\Jc$ with \eqref{eq:korn-inequality-variation}. Existence then follows from \cite[Theorem 2, Section 8.2]{evans1998} due to the convexity of $\Jc$ \cite{chen2013}. We  may prove existence when $R_V \neq \{0\}$ and \eqref{eq:compatibility-condition} holds with the method used in \cite[1.II]{fichera1973}. If a minimising sequence $(\bu_n)$ in $\Ko$ contains a bounded subsequence, then we may extract a subsequence that converges weakly to a point $\bu$. By the weak lower semicontinuity of $\Jc$, we can then show that $\bu$ minimises $\Jc$. We must therefore show that a minimising sequence $(\bu_n)$ has a bounded subsequence. Assume it does not; we must then have that $\nrm{\bu_n}_V\to\infty$. From inequality \eqref{eq:korn-inequality-variation} we deduce that 
	\begin{align}\label{eq:proof-well-pos-1}
		\nrm{\bbQ\bu_n}^r_V \lesssim \Jc(\bu_n) + \langle F, \bu_n \rangle_V.
	\end{align}
	We define $\bw_n = \bu_n/\nrm{\bu_n}_V$ and deduce from \eqref{eq:proof-well-pos-1} that
	\begin{align*}
		\nrm{\bbQ\bw_n}^r_V \lesssim \frac{1}{\nrm{\bu_n}_V^r} \Jc(\bu_n) + \frac{\nrm{F}_{V^\ast}}{\nrm{\bu_n}^{r-1}_V}.
	\end{align*}	
	Hence, $\bbQ\bw_n\to 0$ as $n\to \infty$. Also, since $\nrm{\bw}_V = 1$, we have that $\bbP \bw_n$ is bounded in $R_V$, so there is a subsequence, which we also denote by $(\bw_n)$, that converges to a $\br\in R_V$. In fact, since $\bbQ\bw_n\to 0$, we have that $\bw_n \to \br$ and therefore $\br \in (\Ko \cap R_V)\setminus \{0\}$. We reach a contradiction when we write \eqref{eq:proof-well-pos-1} as
	\begin{align*}
		\nrm{\bu_n}_V^{r-1}\nrm{\bbQ\bw_n}^r_V \lesssim \frac{1}{\nrm{\bu_n}_V} \left(\Jc(\bu_n) - \Jc(0)\right) + \langle F, \bw_n \rangle_V.
	\end{align*}
	and observe that the $\limsup$ of the left-hand side is strictly positive, while the $\liminf$ of the right-hand side is strictly negative due to the compatibility condition \eqref{eq:compatibility-condition}.
	
	Regarding the uniqueness of solutions, if $\bu_1$ and $\bu_2$ are two minimisers for $\Jc$ in $\Ko$, it follows that 
	\begin{align*}
		\langle A\bu_1 - A\bu_2, \bu_1 - \bu_2 \rangle_V + \langle G\bu_1 - G\bu_2, \bu_1 - \bu_2 \rangle_V \leq 0.
	\end{align*}
	Since the operator $A + G$ is monotone, the inequality above must be an equality and using \eqref{eq:korn-inequality-variation} we deduce that $\nrm{\bbQ(\bu_1 - \bu_2)}_V = 0$. Therefore, if $\bu\in \Ko$ minimises $\Jc$, any other minimiser must be of the form $\bu + \br$ with $\br\in R_V$. If $R_V = \{0\}$, we see that any solution must be unique. On the other hand, if $R_V \neq \{0\}$, since $\Jc(\bu + \br) = \Jc(\bu) + F(\br)$, the function $\br\in R_V$ must also satisfy $F(\br) = 0$. Moreover, whenever \eqref{eq:compatibility-condition} holds, we have that $F(\br) = 0$ if and only if $\br = 0$ because $\dim{R_V}=1$. As a result, $\br = 0$ and the solution is unique.
	
	For the converse statement, assume $R_V \neq \{0\}$ and let $(\bu,p,\lambda)$ solve \eqref{eq:mixed-continuous}. Then, from \eqref{eq:mixed-continuous-V} we deduce that 
	\begin{align*}
		\langle F, \br \rangle_V = - \langle \lambda, \gamma_n\br \rangle_\Sigma \leq 0 \quad \forall \br \in K\cap R_V.
	\end{align*}
	
	To prove \eqref{eq:bound-solution}, we first note that \eqref{eq:mixed-continuous-V} and \eqref{eq:korn-inequality-variation} lead to 
	\begin{align}\label{eq:proof-well-pos-bound-Qu}
		\nrm{\bbQ\bu}_V^r \lesssim \langle A\bu + G\bu, \bu \rangle_V = \langle F, \bu \rangle_V.
	\end{align}
	If $R_V = \{ 0 \}$, we have that $\nrm{\bu}_V^{r-1} \lesssim \nrm{F}_{V^\ast}$. If $R_V \neq \{ 0 \}$ and \eqref{eq:compatibility-condition} holds, then $\langle F, \bbP \bu \rangle_V \leq 0$ and we find that $\nrm{\bbQ \bu}_V^{r-1} \lesssim \nrm{F}_{V^\ast}$. By using the inf-sup conditions and H\"older's inequality, we can establish the bounds
	\begin{align}\label{eq:proof-well-pos-bound-lambda-p}
		\nrm{\lambda}_{\Sigma'} + \nrm{p}_Q \lesssim \nrm{\bbQ\bu}_V^{r-1} + \nrm{F}_{V^\ast} \lesssim \nrm{F}_{V^\ast}.
	\end{align}
	We then use \eqref{eq:inequality-F} to show that
	\begin{align}\label{eq:proof-well-pos-bound-Pu}
		\nrm{\bbP\bu}_V\lesssim - \langle F, \bbP\bu\rangle_V = \langle \lambda, \gamma_n(\bbQ\bu) \rangle_\Sigma \lesssim \nrm{F}_{V^\ast}^{r'}.
	\end{align}
	We finally establish the bound \eqref{eq:bound-solution} by putting together \eqref{eq:proof-well-pos-bound-Qu}, \eqref{eq:proof-well-pos-bound-lambda-p}, and \eqref{eq:proof-well-pos-bound-Pu}, and noting that $\nrm{\bu}_V \leq \nrm{\bbP\bu}_V + \nrm{\bbQ\bu}_V$.
\end{proof}


\section{Abstract discretisation}\label{sec:abstract-discretisation}

In this section we propose an abstract discretisation of the mixed system \eqref{eq:mixed-continuous} built in terms of a collection of finite dimensional spaces satisfying certain key properties. We can then introduce a discrete system analogous to \eqref{eq:mixed-continuous} and investigate the conditions under which we have a unique solution. Then, we prove Lemmas \ref{lemma:error-estimates-rigid-error}, \ref{lemma:error-Dv}, and \ref{lemma:error-ps}, which establish upper bounds for the errors of the discrete solutions. 

\subsection{The discrete mixed formulation}

For each parameter $h > 0$, let $V_h\subset V$, $Q_h\subset Q$ and $\Sigma_h\subset L^2(\Gamma_b)$ be finite dimensional subspaces. We also assume that $R_V \subset V_h$ to avoid the need of introducing discrete compatibility conditions. We define the discrete convex sets
\[
	\Lambda_h = \left\lbrace \mu_h\in\Sigma_h : \mu_h \leq 0 \text{ on $\Gamma_b$} \right\rbrace,
\]
and
\[
	K_h = \left\lbrace \bv_h\in V_h : \langle \mu_h, \gamma_n \bv_h \rangle_\Sigma \geq 0 \quad \forall \mu_h\in\Lambda_h \right\rbrace.
\]
An immediate consequence of the definitions of $\Lambda_h$ and $K_h$ is that $\Lambda_h \subset \Lambda$ but $K_h \not\subset K$ unless $\gamma_n(V_h) \subset \Sigma_h$. By the assumption $R_V \subset V_h$ and the fact that $R_V$ is given by purely vertical translations whenever $R_V \neq \{ 0\}$, we have that $K \cap R_V = K_h \cap R_V$.

The discrete analogue of the variational inequality \eqref{eq:vi-continuous} is: find $(\bu_h,p_h)\in K_h\times Q_h$ such that
\begin{align}\label{eq:vi-discrete}
	\langle A\bu_h + G\bu_h - Bp_h - F, \bv_h - \bu_h \rangle_V + \langle Bq_h, \bu_h \rangle_V \geq  0 \quad \forall (\bv_h,q_h)\in K_h\times Q_h.
\end{align}
This discrete variational inequality can be written as a mixed problem by introducing a Lagrange multiplier. This results in the discrete mixed formulation that is the counterpart of \eqref{eq:mixed-continuous}: find $(\bu_h,p_h,\lambda_h)\in V_h\times Q_h\times \Lambda_h$ such that
\begin{subequations}\label{eq:mixed-discrete}
\begin{align}
	\langle A\bu_h + G\bu_h - Bp_h - F, \bv_h \rangle_V  - \langle \lambda_h, \gamma_n \bv_h\rangle_\Sigma &= 0 && \forall \bv_h \in V_h, \label{eq:mixed-discrete-V}\\
	 \langle Bq_h, \bu_h \rangle_V &= 0 && \forall q_h \in Q_h, \\
	\langle \mu_h - \lambda_h, \gamma_n \bu_h\rangle_\Sigma &\geq 0 && \forall \mu_h \in \Sigma_h. \label{eq:mixed-discrete-S}
\end{align}
\end{subequations}

An advantage of using a mixed formulation at the discrete level is that we explicitly enforce a discrete version of the contact conditions \eqref{eq:main-bc-contact}. Just as in \eqref{eq:weak-contact-conditions}, it is possible to show that the conditions $\lambda_h\in\Lambda_h$ and \eqref{eq:mixed-discrete-S} are equivalent to
\begin{align}\label{eq:discrete-contact-conditions}
	\langle \mu_h, \gamma_n\bu_h\rangle_\Sigma \geq 0 \quad \forall \mu_h\in\Lambda_h, \quad \lambda_h\in\Lambda_h \quad \text{and} \quad \langle \lambda_h, \gamma_n\bu_h \rangle_\Sigma = 0.
\end{align}

In order to state a minimisation problem equivalent to \eqref{eq:mixed-discrete}, we must introduce the subspace of $V_h$ of discretely divergence-free functions and the discrete convex set $\Ko_h$:
\[
	\Vo_h = \left\lbrace \bv_h \in V_h : b(\bv_h,q_h) = 0 \quad \forall q_h\in Q_h \right\rbrace \quad \text{and} \quad \Ko_h = \Vo_h \cap K_h.
\]
Then, the discrete mixed problem \eqref{eq:mixed-discrete} is equivalent to the minimisation over $\Ko_h$ of the functional $\Jc:V \to \RR$ defined in \eqref{eq:Jc-functional}, provided that two discrete inf-sup conditions hold. For $V_{b,h} = V_h \cap V_b$, these discrete conditions can be stated as 
\begin{align}
	\sup_{\bv_h\in V_{b,h}} \frac{\langle Bq_h, \bv_h\rangle_V}{\nrm{\bv_h}_V} & \gtrsim \nrm{q_h}_Q \quad \forall q_h\in Q_h, \label{eq:inf-sup-VQh}\\
	\sup_{\bv_h\in V_{h}} \frac{\langle \mu_h, \gamma_n\bv_h\rangle_\Sigma}{\nrm{\bv_h}_V} &\gtrsim \nrm{\mu_h}_\Sigma \quad \forall \mu_h\in \Sigma_h. \label{eq:inf-sup-VSh}
\end{align}
When the conditions \eqref{eq:inf-sup-VQh} and \eqref{eq:inf-sup-VSh} hold, then \eqref{eq:vi-discrete}, \eqref{eq:mixed-discrete} and the minimisation of $\Jc$ over $\Ko$ are equivalent problems. The proofs for such equivalences require the same arguments as the proofs presented in Appendix \ref{app:equivalence-forms}. If $\Jc$ admits a unique minimiser over $\Ko_h$, the discrete inf-sup conditions guarantee a unique solution for \eqref{eq:mixed-discrete} and set constraints on the choice of spaces $V_h$, $Q_h$ and $\Sigma_h$ used when approximating solutions of \eqref{eq:mixed-continuous}. As in the continuous case, the well-posedness of \eqref{eq:mixed-discrete} requires the compatibility condition \eqref{eq:compatibility-condition} to hold. The theorem below can be proved in the same way as Theorem \ref{thm:continuous-well-posed}.

\begin{theorem}\label{thm:discrete-well-posed}
	Assume that the discrete inf-sup conditions \eqref{eq:inf-sup-VQh} and \eqref{eq:inf-sup-VSh} hold. If $R_V = \{0\}$, then a solution to \eqref{eq:mixed-discrete} exists and is unique. If $R_V \neq \{0\}$, then there is a unique solution to \eqref{eq:mixed-discrete} if the compatibility condition \eqref{eq:compatibility-condition} holds. Conversely, if $R_V \neq \{0\}$ and a solution exists, \eqref{eq:compatibility-condition-weak} must hold. The solution of \eqref{eq:mixed-continuous} is bounded from above independently of $h$, provided \eqref{eq:compatibility-condition} holds when $R_V \neq \{0\}$.
\end{theorem}

\subsection{Upper bounds for the velocity error}

An important tool presented in \cite{belenki2012,hirn2013} for establishing error estimates for non-Newtonian flows is the use of the function $\bF$. Here, for ease of notation, we denote by $\bF$ an operator that acts on both $\RR^{2\times 2}$ and $\RR^2$ by
\begin{align}\label{eq:def-F}
	\bF(\mathrm{A}) = |\mathrm{A}|^{\frac{r-2}{2}} \mathrm{A} \quad \text{for $\mathrm{A}\in\RR^{2\times 2}$ or $\mathrm{A}\in\RR^2$}.
\end{align}
This operator is closely related to the operators $A$ and $G$. Let the operator $\bE:V\times V\to \RR$ be given by
\begin{align*}
	\bE(\bu,\bv) = \nrm{\bF(\bD\bu) - \bF(\bD\bv)}^2_{L^2(\Omega)} + \nrm{\bF(\bT\bu) - \bF(\bT\bv)}^2_{L^2(\Gamma_b)}.
\end{align*}
We then have that 
\begin{equation}\label{eq:F-ineq-1}
	\bE(\bu,\bv)\sim \langle A\bu - A\bv, \bu - \bv \rangle_V +  \langle G\bu - G\bv, \bu - \bv \rangle_V 
\end{equation}
for all $\bu,\bv\in V$. The following variation of Young's inequality, 
\begin{align}\label{eq:F-ineq-2}
	\langle A\bu - A\bv, \bu - \bw \rangle_V + \langle G\bu - G\bv, \bu - \bw \rangle_V \leq \varepsilon \bE(\bu,\bv) + c_\varepsilon \bE(\bu,\bw),
\end{align}
is valid for any $\bu,\bv,\bw \in V$ and $\varepsilon > 0$, with the constant $c_\varepsilon > 0$ depending on $\varepsilon$. Additionally, the inequalities
\begin{align}
	\nrm{\bD\bv - \bD\bw}^2_{L^r(\Omega)}  &\lesssim \nrm{\bF(\bD\bv) - \bF(\bD\bw)}^2_{L^2(\Omega)} \nrm{|\bD\bv| + |\bD\bw|}_{L^r(\Omega)}^{2-r}, \label{eq:F-ineq-3-1} \\
	\nrm{\bT\bv - \bT\bw}^2_{L^r(\Gamma_b)} &\lesssim \nrm{\bF(\bT\bv) - \bF(\bT\bw)}^2_{L^2(\Gamma_b)} \nrm{|\bT\bv| + |\bT\bw|}_{L^r(\Gamma_b)}^{2-r}, \label{eq:F-ineq-3-2}
\end{align}
hold for any $\bv,\bw\in \bW^{1,r}(\Omega)$. A proof for inequalities \eqref{eq:F-ineq-1} and \eqref{eq:F-ineq-3-1} can be found in \cite[Lemmas 2.3, 2.4]{hirn2013}, and \cite[Lemma 2.7]{belenki2012} for \eqref{eq:F-ineq-2}, for the case without friction. The presence of the operator $G$ requires a version of \cite[Lemmas 2.3, 2.4]{hirn2013} and \cite[Lemma 2.7]{belenki2012} stated in terms of vectors in $\RR^d$. Since these results are based on algebraic inequalities for matrices, the extension to vectors in $\RR^d$ can be proved by considering diagonal matrices.

By applying the triangle inequality, \eqref{eq:korn-inequality-variation} and \eqref{eq:F-ineq-3-1}-\eqref{eq:F-ineq-3-2}, the velocity error can be decomposed into two components as
\begin{align}\label{eq:error-decomposition}
	\nrm{\bu - \bu_h}_V \lesssim \nrm{\bbP(\bu - \bu_h)}_V + \bE(\bu,\bu_h).
\end{align}
For the first term on the right of \eqref{eq:error-decomposition}, which represents the rigid component of the error, we present the following result:

\begin{lemma}\label{lemma:error-estimates-rigid-error}
	Assume that $R_V \neq \{0\}$ and that the compatibility condition \eqref{eq:compatibility-condition} holds. Let $(\bu,p,\lambda)\in V\times Q\times \Lambda$ be the solution to \eqref{eq:mixed-continuous} and $(\bu_h,p_h,\lambda_h)\in V_h\times Q_h\times \Lambda_h$ to \eqref{eq:mixed-discrete}. Then,
	\begin{align}\label{eq:error-vel-rigid}
		\nrm{\bbP(\bu - \bu_h)}_V \lesssim \bE(\bu,\bu_h) + \nrm{\lambda -\mu_h}_{\Sigma'} \quad \forall\,\mu_h\in \Lambda_h.
	\end{align}
\end{lemma}

\begin{proof}
	Under the assumption that $R_V \neq \{0\}$, we either have that $\bbP(\bu - \bu_h)\in R_V\cap K$ or $- \bbP(\bu - \bu_h)\in R_V\cap K$. If $\bbP(\bu - \bu_h)\in R_V\cap K$, then inequality \eqref{eq:inequality-F} and the continuous mixed system \eqref{eq:mixed-continuous} allow us to write 
	\begin{equation*}
		\nrm{\bbP(\bu - \bu_h)}_V \lesssim - \langle F, \bbP(\bu - \bu_h)\rangle_V = \langle \lambda, \gamma_n(\bbP(\bu - \bu_h))\rangle_\Sigma,
	\end{equation*}
	where the equality follows from $\langle A\bu + G\bu - Bp, \bbP(\bu-\bu_h)\rangle_V = 0$. Then, by noting that $\bbP(\bu - \bu_h) = - \bbQ(\bu - \bu_h) + \bu - \bu_h$ and $\langle\lambda,\gamma_n\bu\rangle_\Sigma = 0$, using inequalities \eqref{eq:korn-inequality-variation} and \eqref{eq:F-ineq-3-1}-\eqref{eq:F-ineq-3-2}, and using the uniform in $h$ boundedness of solutions to \eqref{eq:mixed-discrete} (see Theorem \ref{thm:discrete-well-posed}), we arrive at 
	\begin{equation}\label{eq:proof-rigid-error-case1}
		\nrm{\bbP(\bu - \bu_h)}_V \lesssim \bE(\bu,\bu_h) - \langle \lambda, \gamma_n \bu_h \rangle_\Sigma.
	\end{equation}
	On the other hand, if $- \bbP(\bu - \bu_h)\in R_V\cap K$, then,  by appealing to \eqref{eq:compatibility-condition} and the discrete mixed system \eqref{eq:mixed-discrete}, 
	\begin{equation*}
		\nrm{\bbP(\bu - \bu_h)}_V \lesssim \langle F, \bbP(\bu - \bu_h)\rangle_V = - \langle \lambda_h, \gamma_n(\bbP(\bu - \bu_h))\rangle_\Sigma.
	\end{equation*}
	Following the same steps as before, we deduce that 
	\begin{equation}\label{eq:proof-rigid-error-case2}
		\nrm{\bbP(\bu - \bu_h)}_V \lesssim \bE(\bu,\bu_h) - \langle \lambda_h, \gamma_n \bu \rangle_\Sigma \leq \bE(\bu,\bu_h),
	\end{equation}
	where the final inequality follows from the fact that $\lambda_h \leq 0$ on $\Gamma_b$ by the definition of $\Lambda_h$. As a result of \eqref{eq:proof-rigid-error-case1} and \eqref{eq:proof-rigid-error-case2}, we have that 
	\begin{equation*}
		\nrm{\bbP(\bu - \bu_h)}_V \lesssim \bE(\bu,\bu_h) + \max{\left\lbrace 0, - \langle \lambda, \gamma_n \bu_h \rangle_\Sigma \right\rbrace }
	\end{equation*}
	in all cases. Finally, given a $\mu_h\in\Lambda_h$, we have 
	\begin{align*}
		- \langle \lambda, \gamma_n \bu_h \rangle_\Sigma \leq \langle \mu_h - \lambda, \gamma_n \bu_h\rangle \lesssim \nrm{\lambda - \mu_h}_{\Sigma'}
	\end{align*}
	because $\langle \mu_h, \gamma_n\bu_h\rangle_\Sigma \geq 0$.
\end{proof}

\begin{remark}
	As mentioned in the introduction, previous analyses of finite element approximations of semicoercive variational inequalities either only consider the error in a seminorm \cite{hlavacek1978} or use indirect arguments to prove the convergence of the approximate solution in the complete norm \cite{bon1988, gwinner1991, spann1994, adly2000,chadli2018}. In these cases, arguments by contradiction involving a sequence of triangulations are used. In Lemma \ref{lemma:error-estimates-rigid-error}, on the other hand, we provide a fully constructive proof for bounding the rigid component of the velocity error from above. This result is a key ingredient in obtaining the error estimates for the finite element scheme presented in the next section. The proof of Lemma \ref{lemma:error-estimates-rigid-error} relies on $\dim{R_V} \leq 1$, which holds for almost all Stokes variational inequalities considered in glaciology \cite{durand2009, pattyn2012, favier2012, stubblefield2021, dediego2022}. 
\end{remark}

The second term on the right of \eqref{eq:error-decomposition} can be bounded from above by using the properties of the operator $\bE$.
	
	\begin{lemma}\label{lemma:error-Dv}
		Let the triples $(\bu,p,\lambda)\in V\times Q\times \Lambda$ and $(\bu_h,p_h,\lambda_h)\in V_h\times Q_h\times \Lambda_h$ be solutions to \eqref{eq:mixed-continuous} and \eqref{eq:mixed-discrete}, respectively. Then
		\begin{equation}\label{eq:error-Dv}
				\bE(\bu,\bu_h) \lesssim \bE(\bu,\bv_h) + \nrm{p - q_h}^2_Q  + \langle \lambda - \lambda_h, \gamma_n(\bv_h - \bu_h) \rangle_\Sigma
		\end{equation}
		holds for all $(\bv_h,q_h)\in \Vo_h\times Q_h$.
	\end{lemma}
	
	\begin{proof}
		From \eqref{eq:mixed-continuous-V} and \eqref{eq:mixed-discrete-V}, we see that, for any $(\bv_h,q_h)\in\Vo_h\times Q_h$, we have
		\begin{equation*}
			\begin{split}
				\langle A\bu - A\bu_h, \bu - \bu_h \rangle_V + \langle G\bu - G\bu_h, \bu - \bu_h \rangle_V = \\
				\langle A\bu - A\bu_h, \bu - \bv_h \rangle_V + \langle G\bu - G\bu_h, \bu - \bv_h \rangle_V \\
				+ \langle B(p - q_h), \bv_h - \bu_h \rangle_V + \langle \lambda - \lambda_h, \gamma_n(\bv_h - \bu_h) \rangle_\Sigma
			\end{split}
		\end{equation*}		
		Using \eqref{eq:F-ineq-1} and \eqref{eq:F-ineq-2}
		\begin{equation*}
			\begin{split}
				\bE(\bu,\bu_h) &\lesssim \varepsilon_1 \bE(\bu,\bu_h) + c_{\varepsilon_1} \bE(\bv,\bv_h) \\
				&+ \langle B(p - q_h), \bv_h - \bu_h \rangle_V + \langle \lambda - \lambda_h, \gamma_n(\bv_h - \bu_h) \rangle_\Sigma
			\end{split}
		\end{equation*}
		for an arbitrary $\varepsilon_1 > 0$. Additionally, by using Young's inequality,
		\begin{align*}
			\langle B(p - q_h), \bv_h - \bu_h \rangle_V \lesssim c_{\varepsilon_2} \nrm{p - q_h}^2_Q + \varepsilon_2 \nrm{\bD(\bu_h - \bv_h)}^2_{L^r(\Omega)}
		\end{align*}
		for any $\varepsilon_2 > 0$. Then, via \eqref{eq:F-ineq-3-1}, and by setting $\varepsilon_1$ and $\varepsilon_2$ sufficiently small, inequality \eqref{eq:error-Dv} is established.
	\end{proof}

	\begin{remark}
		If the pair $V_h\times Q_h$ is divergence free in the sense that $\langle Bq_h, \bw_h\rangle_V= 0$ for all $q_h\in Q_h$ implies that $\nabla\cdot \bw_h = 0$, then the term $\nrm{p - q_h}^2_Q$ in inequality \eqref{eq:error-Dv} can be removed.
	\end{remark}

\subsection{Upper bounds for the pressure and Lagrange multiplier errors}

	We finalise the analysis of the abstract discretisation by bounding the errors for the pressure and the Lagrange multiplier from above.
	
	\begin{lemma}\label{lemma:error-ps}
		Assume that the discrete inf-sup conditions \eqref{eq:inf-sup-VQh} and \eqref{eq:inf-sup-VSh} hold. Let $(\bu,p,\lambda)\in V\times Q\times \Lambda$ be the solution of \eqref{eq:mixed-continuous} and $(\bu_h,p_h,\lambda_h)\in V_h\times Q_h\times \Lambda_h$ of \eqref{eq:mixed-discrete}. Then
		\begin{align}
			\nrm{p - p_h}_Q &\lesssim\bE(\bu,\bu_h)^{1/r'} + \nrm{p - q_h}_Q, \label{eq:error-Q}\\
			\nrm{\lambda - \lambda_h}_{\Sigma'} &\lesssim \bE(\bu,\bu_h)^{1/r'} + \nrm{p - q_h}_Q + \nrm{\lambda - \mu_h}_{\Sigma'}, \label{eq:error-S}
		\end{align}
		for all $q_h\in Q_h$ and $\mu_h\in\Sigma_h$.
	\end{lemma}
	
	\begin{proof}
		Since $Q_h$ and $\Sigma_h$ are subsets of $Q$ and $\Sigma$ respectively, we can obtain the following equality from \eqref{eq:mixed-continuous-V} and \eqref{eq:mixed-discrete-V}:
		\begin{align}\label{eq:proof-error-QS}
			\begin{split}
				\langle A\bu - A\bu_h, \bv_h \rangle_V + \langle G\bu - G\bu_h, \bv_h \rangle_V = \\
				\langle B(p - p_h), \bv_h \rangle_V + \langle \lambda - \lambda_h, \gamma_n \bv_h \rangle_\Sigma \quad \forall \bv_h\in V_h.		
			\end{split}
		\end{align}
		The inf-sup condition \eqref{eq:inf-sup-VQh} for the pressure space holds over the space $V_{b,h}\subset V_h$ of vector fields with a normal component vanishing on $\Gamma_b$. For $\bv_h\in V_{b,h}$, from equation \eqref{eq:proof-error-QS} we derive
		\[
			\langle B(p_h - q_h), \bv_h \rangle_V  = \langle A\bu - A\bu_h, \bv_h \rangle_V + \langle B(p - q_h), \bv_h \rangle_V.
		\]
		From the inf-sup condition \eqref{eq:inf-sup-VQh} it follows that
		\begin{align}\label{eq:proof-error-QS-1}
			\nrm{p_h - q_h}_Q \lesssim \sup_{\bv_h\in V_{b,h}}{\left( \frac{\langle A\bu - A\bu_h, \bv_h \rangle_V}{\nrm{\bv_h}_V} \right)} + \nrm{p - q_h}_Q.
		\end{align}
		Then, by H\"older's inequality and \cite[Lemma 2.4]{hirn2013}, we have that 
		\begin{align}\label{eq:proof-error-QS-2}
			\sup_{\bv_h\in V_{b,h}}{\left( \frac{\langle A\bu - A\bu_h, \bv_h \rangle_V}{\nrm{\bv_h}_V} \right)} \lesssim \nrm{\bF(\bD\bu) - \bF(\bD\bu_h)}_{L^2(\Omega)}^{2/r'}.
		\end{align}
		Finally, \eqref{eq:error-Q} follows by applying the triangle inequality to $\nrm{p - p_h}_Q$ and using \eqref{eq:proof-error-QS-1} and \eqref{eq:proof-error-QS-2}. The bound \eqref{eq:error-S} follows in the same way. However, in this case, the inf-sup condition \eqref{eq:inf-sup-VSh} is set over the whole space $V_h$, so the friction term in \eqref{eq:proof-error-QS} does not vanish. Following the argument used in \cite[Lemma 2.4]{hirn2013}, we can show that  
		\[
			\nrm{|\bT\bu|^{r-2}\bT\bu - |\bT\bu_h|^{r-2}\bT\bu_h}_{L^2(\Gamma_b)} \lesssim \nrm{\bF(\bT\bu) - \bF(\bT\bu_h)}^{2/r'}_{L^2(\Gamma_b)}
		\]
		and therefore
		\[
			\sup_{\bv_h\in V_h}{\left( \frac{\langle G\bu - G\bu_h, \bv_h \rangle_V}{\nrm{\bv_h}_V} \right)} \lesssim \nrm{\bF(\bT\bu) - \bF(\bT\bu_h)}^{2/r'}_{L^2(\Gamma_b)}.
		\]
		\end{proof}

        Lemmas \ref{lemma:error-estimates-rigid-error}, \ref{lemma:error-Dv}, and \ref{lemma:error-ps} give discretisation error estimates in terms of best approximation results. To derive a convergence result, we require bounds on these best approximations. We discuss this in the context of a finite element discretisation in the next section.

\section{A finite element scheme}\label{sec:fe-scheme}

We now consider a particular finite element discretisation of the mixed problem \eqref{eq:mixed-continuous}. We introduce a non-degenerate (in the sense of \cite[Definition 4.4.13]{brenner2007}) sequence of triangulations $\Tc_h$, where $h>0$ denotes the maximum cell diameter in $\Tc_h$. The set of edges in $\Tc_h$ is denoted by $\Ec_h$; we assume that every edge $e\in \Ec_h$ in $\partial\Omega$ is either in $\overline{\Gamma}_s$, $\overline{\Gamma}_b$ or $\overline{\Gamma}_d$. We write $\Ec_h(\Gamma_s)$ and $\Ec_h(\Gamma_b)$ to denote the edges in $\overline{\Gamma}_s$ and $\overline{\Gamma}_b$ respectively. Associated to each $\Tc_h$ are the finite element spaces $V_h$, $Q_h$ and $\Sigma_h$, defined by
\begin{subequations}\label{eq:fem-spaces}
	\begin{align}
		V_h &= \left\lbrace \bv_h\in \bCc(\Omega) : \bv_h |_c \in \bPc_2(c) \quad \forall c \in\Tc, \quad \bv_h\cdot\bn = 0 \quad \text{on $\Gamma_d$} \right\rbrace, \label{eq:Vh}\\
		Q_h &= \left\lbrace q_h\in L^2(\Omega) : q_h|_c \in \Pc_0(c) \quad \forall c \in \Tc \right\rbrace, \label{eq:Qh}\\
		\Sigma_h &= \left\lbrace \mu_h\in L^2(\Gamma_b) : \mu_h|_e \in \Pc_0(e) \quad \forall e \in \Ec(\Gamma_b) \right\rbrace. \label{eq:Sh}
	\end{align}
\end{subequations}

\subsection{Analysis of the scheme}

The first step in analysing this discretisation is to investigate whether the discrete mixed problem \eqref{eq:mixed-discrete} is well-posed for this choice of $V_h\times Q_h\times \Sigma_h$, subject to the compatibility condition \eqref{eq:compatibility-condition}. Specifically, we must verify the discrete inf-sup conditions \eqref{eq:inf-sup-VQh} and \eqref{eq:inf-sup-VSh}. The pair $V_h\times Q_h$ is well-known to satisfy \eqref{eq:inf-sup-VQh}, see \cite[Proposition 8.4.3]{boffi2013} for the case of $r=2$; the general case $r \in [1,\infty]$ follows from the same arguments by using the interpolation operator $\pi_V$ discussed in Appendix \ref{app-subsec:interpolation-V}. A proof for \eqref{eq:inf-sup-VSh} is presented below using a similar argument to the one presented in \cite[Proposition 3.3]{coorevits2002}.

\begin{lemma}\label{lemma:inf-sup-VSh}
	The finite element pair $V_h$ and $\Sigma_h$ defined in \eqref{eq:fem-spaces} is inf-sup stable in the sense of \eqref{eq:inf-sup-VSh}.
\end{lemma}

\begin{proof}
	Let $\mu_h\in\Sigma_h$. By the Hahn-Banach theorem, there is a $\psi\in \Sigma$ such that $\langle \mu_h, \psi\rangle_\Sigma = \nrm{\mu_h}_{\Sigma'}$ and $\nrm{\psi}_\Sigma = 1$. In Appendix \ref{app-subsec:extension} we construct an extension operator $\Phi:\Sigma \to V_h$ which is bounded uniformly with respect to $h$, i.e.~$\nrm{\Phi\psi}_V \lesssim \nrm{\psi}_\Sigma$, and
	\[
		\langle \mu_h, \gamma_n (\Phi\psi)\rangle_\Sigma = \langle \mu_h, \psi\rangle_\Sigma \quad \forall \mu_h\in \Sigma_h,
	\]
	for all $\psi \in \Sigma$. Then
	\begin{align*}
		\nrm{\mu_h}_{\Sigma'} = \frac{\langle \mu_h, \psi\rangle_\Sigma}{\nrm{\psi}_\Sigma} \lesssim \frac{\langle \mu_h, \gamma_n (\Phi\psi)\rangle_\Sigma}{\nrm{\Phi\psi}_V} \leq \sup_{\bv_h\in V_h} \frac{\langle\mu_h, \gamma_n\bv_h\rangle_\Sigma}{\nrm{\bv_h}_V}
	\end{align*}
	and the result follows.
\end{proof}

We end this section with a discussion on the approximability of the mixed system \eqref{eq:mixed-continuous}. We show that the approximate solutions $(\bu_h,p_h,\lambda_h)$ converge to the exact solutions of \eqref{eq:mixed-continuous} as $h\to 0$ under a regularity condition, and we establish a rate of convergence for these approximations. 

\begin{theorem}\label{thm:convergence}
	Assume that the compatibility condition \eqref{eq:compatibility-condition} holds whenever $R_V \neq \{0\}$. Let the triple $(\bu,p,\lambda)\in V\times Q\times \Lambda$ be the solution to \eqref{eq:mixed-continuous} and $(\bu_h,p_h,\lambda_h)\in V_h\times Q_h\times \Lambda_h$ to \eqref{eq:mixed-discrete}. Additionally, assume that $(\bu,p,\lambda)\in W^{2,r}(\Omega) \times W^{1,r'}(\Omega) \times W^{1 - 1/r',r'}(\Gamma_b)$ and $\bF(\bD\bu)\in \bW^{1,2}(\Omega)$ and $\bF(\bT\bu)\in \bW^{1,2}(\Gamma_b)$. Then
	\begin{subequations}\label{eq:error-estimate}
	\begin{align}
		\nrm{\bu - \bu_h}_V &\lesssim h, \label{eq:error-estimate-V} \\
		\nrm{p - p_h}_Q + \nrm{\lambda - \lambda_h}_{\Sigma'} &\lesssim h^{2/r'} + h . \label{eq:error-estimate-QS}
	\end{align}
	\end{subequations}

\end{theorem}

\begin{proof}
	We proceed by first finding a suitable upper bound for the term involving the Lagrange multiplier in \eqref{eq:error-Dv}. Since $\langle \lambda_h,\gamma_n\bu_h\rangle_\Sigma = 0$ and $\langle \mu_h,\gamma_n\bv_h\rangle_\Sigma\geq 0$ for all $(\bv_h,\mu_h)\in K_h\times \Lambda_h$, one can show that
	\begin{align}\label{eq:proof-error-estimates-1}
		\begin{split}
			\langle \lambda - \lambda_h, \gamma_n(\bv_h - \bu_h)\rangle_\Sigma & \leq \langle \lambda - \mu_h, \gamma_n (\bv_h - \bu) \rangle_\Sigma \\
			&+ \langle \lambda - \mu_h, \gamma_n (\bu - \bu_h) \rangle_\Sigma + \langle \mu_h, \gamma_n\bv_h \rangle_\Sigma
		\end{split}
	\end{align}
	for all $(\bv_h,\mu_h)\in K_h\times \Lambda_h$. By applying Young's inequality in \eqref{eq:proof-error-estimates-1} and using inequalities \eqref{eq:error-decomposition}, \eqref{eq:error-vel-rigid}, and \eqref{eq:error-Dv}, we arrive at 
	\begin{align}\label{eq:proof-error-estimates-2}
		\nrm{\bu - \bu_h}_V \lesssim \bE(\bu,\bv_h)^{1/2} + \nrm{p - q_h}_Q + \nrm{\lambda - \mu_h}_{\Sigma'} + \langle \mu_h, \gamma_n\bv_h \rangle_\Sigma
	\end{align}
	for all $(\bv_h,q_h,\mu_h)\in \Ko_h\times Q_h \times \Lambda_h$.
	
	Let $\pi_V:V\to V_h$ by the interpolation operator introduced in Appendix \ref{app-subsec:interpolation-V}. Additionally, let $\pi_Q:Q\to Q_h$ and $\pi_\Sigma:L^2(\Gamma_b)\to \Sigma_h$ be standard interpolation operators onto the space of piecewise-constant functions. We refer to \cite{ern2017} and the results in Appendix \ref{app-subsec:interpolation} for proofs of optimal interpolation error estimates in the $Q$ and $\Sigma'$ norms. From the properties of these interpolation operators it follows that $\pi_V\bu\in \Ko_h$ and $\pi_\Sigma\lambda\in\Lambda_h$. Additionally, we have that $\int_e \bu\cdot\bn\,\dd s = \int_e (\pi_V\bu)\cdot\bn\,\dd s$ for all $e\in \Ec$, so 
	\begin{align*}
		\langle \pi_\Sigma\lambda, \gamma_n\pi_V\bu \rangle_\Sigma = \langle \lambda, \pi_\Sigma(\gamma_n\bu)\rangle_\Sigma = \langle \lambda, \pi_\Sigma(\gamma_n\bu) - \gamma_n\bu\rangle_\Sigma.
	\end{align*}
	Since $\langle \pi_\Sigma\lambda, \pi_\Sigma(\gamma_n\bu) - \gamma_n\bu\rangle_\Sigma = 0$, we have that 
	\begin{align}\label{eq:proof-error-estimates-3}
		\langle \pi_\Sigma\lambda, \gamma_n\pi_V\bu \rangle_\Sigma = \langle \lambda - \pi_\Sigma\lambda, \pi_\Sigma(\gamma_n\bu) - \gamma_n\bu\rangle_\Sigma.
	\end{align}
	Therefore, by setting $\bv_h = \pi_V\bu$, $q_h = \pi_Qp$, and $\mu_h = \pi_\Sigma\lambda$ in \eqref{eq:proof-error-estimates-2} and using \eqref{eq:proof-error-estimates-3}, we can show that 
	\begin{align}
		\nrm{\bu - \bu_h}_V \lesssim \bE(\bu,\pi_V\bu)^{1/2} + \nrm{p - \pi_Qp}_Q + \nrm{\lambda - \pi_\Sigma\lambda}_{\Sigma'} + \nrm{\gamma_n\bu - \pi_\Sigma(\gamma_n\bu)}_{\Sigma}.
	\end{align}	
	We then establish \eqref{eq:error-estimate-V} with the approximation properties of the interpolation operators presented in Appendices \ref{app-subsec:interpolation} and \ref{app-subsec:interpolation-V} for $\pi_\Sigma$ and $\pi_V$ respectively, and \cite{ern2017} for $\pi_Q$. The estimate \eqref{eq:error-estimate-QS} then follows from Lemma \ref{lemma:error-ps}.
\end{proof}

\begin{remark}
	The velocity and pressure error estimates coincide with those obtained in \cite[Theorem 2.14]{belenki2012} and in \cite[Theorem 3.1]{hirn2013} for the $r$-Stokes system without contact or friction boundary conditions. This indicates that these boundary conditions and the Lagrange multiplier do not reduce the order of convergence. This may be due to the use of piecewise constant elements for $\Sigma_h$. In \cite{coorevits2002}, a proof with non-optimal convergence rates is presented for the case when continuous piecewise quadratic polynomials are used for the Lagrange multiplier.
\end{remark}

\subsection{Discrete algebraic formulation}

We now present an algebraic counterpart of \eqref{eq:mixed-discrete} using the finite element spaces specified in \eqref{eq:fem-spaces} in terms of matrices and vectors. Let $V_h = \Span{\{\bv_i\}}_{i=1}^{N_v}$, $Q_h = \Span{\{q_j\}}_{j=1}^{N_q}$ and $\Sigma_h = \Span{\{\mu_k\}}_{k=1}^{N_\mu}$, where $N_v = \dim{V_h}$, $N_q = \dim{Q_h}$ and $N_\mu = \dim{\Sigma_h}$. For the functions $(\bu_h, p_h, \lambda_h) \in V_h\times Q_h\times \Sigma_h$, we write $\mathbf{u}$, $\mathbf{p}$ and $\blambda$ for the vectors containing the respective degrees of freedom (DoFs) in $\RR^{N_v}$, $\RR^{N_q}$ and $\RR^{N_\mu}$. In order to write an algebraic counterpart of \eqref{eq:mixed-discrete-S}, we need to introduce the discrete normal trace operator 
\begin{align*}
	\boldsymbol{\gamma_n} : \RR^{N_v} \to \RR^{N_\mu}
\end{align*}
that returns the average normal components of a vector $\bv_h\in V_h$ along the edges on $\Gamma_b$. That is, for each $i \in \{1,2,...,N_\mu\}$
\begin{align*}
	(\boldsymbol{\gamma_n} \mathbf{v})_i = \frac{1}{|e_i|} \int_{e_i} \bv_h\cdot\bn\,\dd s,
\end{align*}
where $e_i \in \Ec(\Gamma_b)$ is the unique edge along $\Gamma_b$ associated to the degree of freedom in $\Sigma_h$ with index $i$. Then, the algebraic counterpart of  \eqref{eq:mixed-discrete} can be written in terms of matrices and vectors as
\begin{subequations}\label{eq:algebraic}
	\begin{align}
		\mathbf{A}_\varepsilon(\mathbf{u}) + \mathbf{G}_\varepsilon(\mathbf{u}) - \mathbf{B} \mathbf{p} - \mathbf{D}\boldsymbol{\lambda} &= \mathbf{f},\label{eq:algebraic-1}\\
		\mathbf{B}^\top\mathbf{u} &= 0, \label{eq:algebraic-2}\\
		\blambda + \mathbf{C}(\blambda,\mathbf{u}) &= 0 \label{eq:algebraic-3}.
	\end{align}
\end{subequations}
Here, we have introduced the matrices $\mathbf{B}\in\RR^{N_v\times N_q}$ and $\mathbf{D}\in\RR^{N_v\times N_\mu}$, the vector $\mathbf{f}\in\RR^{N_v}$ and the nonlinear operators $\mathbf{A}_\varepsilon:\RR^{N_v}\to\RR^{N_v}$, $\mathbf{G}_\varepsilon:\RR^{N_v}\to\RR^{N_v}$ and $\mathbf{C}:\RR^{N_\mu}\times\RR^{N_v}\to \RR^{N_\mu}$. The matrices are given by the elements $\mathbf{B}_{ij} = \langle Bq_j, \bv_i\rangle_V$ and $\mathbf{D}_{ij} = \langle \mu_j, \gamma_n\bv_i \rangle_{\Sigma}$ and the vector by $\mathbf{f}_i = \langle F, \bv_i \rangle_V$. The nonlinear operators are defined as
\begin{align}
	[\mathbf{A}_\varepsilon(\mathbf{u})]_i &= \int_\Omega \alpha \left(\varepsilon + |\bD\bu|\right)^{r-2} (\bD \bu : \bD \bv_i)\,\dd x, \label{eq:A-discrete}\\
	[\mathbf{G}_\varepsilon(\mathbf{u})]_i &= \int_\Omega \tau \left(\varepsilon + |\bT\bu|\right)^{r-2} (\bT \bu \cdot \bT \bv_i)\,\dd x, \label{eq:G-discrete}\\
	\mathbf{C}(\blambda,\mathbf{u}) &=  \max{\left\lbrace 0, -\blambda + c (\boldsymbol{\gamma_n}\mathbf{u}) \right\rbrace}, \label{eq:complementarity}
\end{align}
for an arbitrary $c>0$ and a regularisation term $\varepsilon > 0$. This regularisation term is commonly used to avoid numerical complications caused when $|\bD\bu|$ or $|\bT\bu|$ is equal to or very close to zero \cite{jouvet2012}. In \eqref{eq:complementarity}, the $\max$ operation is understood to be carried out componentwise on each of the elements in the vector $-\blambda + c (\boldsymbol{\gamma_n}\mathbf{u})\in\RR^{N_\mu}$. The use of the operator $\mathbf{C}$ in \eqref{eq:complementarity} is a common way of expressing contact conditions. A particular advantage is that the nonlinear system \eqref{eq:algebraic} can be solved with a semi-smooth Newton method that enjoys superlinear convergence in a neighbourhood of the solution \cite{hintermuller2002}. Equation \eqref{eq:algebraic-3} is equivalent to \eqref{eq:mixed-discrete-S} whenever $\Sigma_h$ is defined as in \eqref{eq:Sh}. By solving \eqref{eq:algebraic-3} we enforce
\begin{align}\label{eq:algebraic-contact}
	\boldsymbol{\gamma_n} \mathbf{u} \leq 0, \quad \boldsymbol{\lambda} \leq 0 \quad \text{and} \quad (\boldsymbol{\gamma_n} \mathbf{u})\cdot \boldsymbol{\lambda} = 0,
\end{align}	
exactly, which is the algebraic equivalent of the discrete contact conditions \eqref{eq:discrete-contact-conditions}.

\subsection{Numerical results}

We present numerical results computed for a Stokes variational inequality with a manufactured solution on the domain $\Omega = (0,1)^2$. Computations of subglacial cavitation with the algorithm presented in this document can be found in \cite{dediego2022}. Subglacial cavitation is a time-dependent problem in which a free surface is evolved with an advection equation. We do this in \cite{dediego2022} by coupling the algorithm presented in this study with a solver for the advection equation.

The manufactured solution considered here is taken from \cite{belenki2012} and is given by
\begin{align}\label{eq:exact-sol}
	\hat{\bu}(\bx) = |\bx|^{\alpha - 1} (x_2, -x_1)^\top, \quad \hat{p}(\bx) = |\bx|^\gamma,
\end{align}
where the parameters $\alpha$ and $\gamma$ are chosen such that $\bu\in \bW^{2,r}(\Omega)$, $p\in W^{1,r'}(\Omega)$, $\bF(\bD\bu)\in \bW^{1,2}(\Omega)$ and $\bF(\bT\bu)\in \bW^{1,2}(\Gamma_b)$ hold. This is ensured whenever $\alpha > 1$ and $\gamma > -1 + \frac{2}{r}$, so we set $\alpha = 1.01$ and $\gamma = -1 + \frac{2}{r} + 0.01$ in order to be critically close to the regularity assumed in Theorem \ref{thm:convergence}.

Contact boundary conditions are enforced on the lower boundary $\{ y = 0\}$. Given the velocity and pressure fields defined in \eqref{eq:exact-sol}, we have that
\begin{align*}
	(\hat{\bu}\cdot\bn)(x_1) = - x_1^{\alpha}, \quad \hat{\lambda}(x_1) = -x_1^\gamma,
\end{align*}
on $\{ y = 0\}$. In order to define the contact boundary conditions in such a way that both the kinematic and dynamic conditions are active, we define the ``obstacles''
\begin{align*}
	\chi(x_1) = \begin{cases}
		(\hat{\bu}\cdot\bn)(x_1) & \text{if $x_1 \leq 0.5$} \\
		-2^{-\alpha} & \text{if $x_1 > 0.5$} \\
	\end{cases},
	\quad
	\rho(x_1) = \begin{cases}
		0 & \text{if $x_1 \leq 0.5$} \\
		\hat{\lambda}(x_1) & \text{if $x_1 > 0.5$} \\
	\end{cases}.
\end{align*}
Then, for this numerical test we solve the $r$-Stokes system \eqref{eq:main} together with the boundary conditions
\begin{subequations}\label{eq:bc-exact}
\begin{align}
	\bu\cdot\bn \leq \chi, \quad \lambda \leq \rho \quad \text{and} \quad (\bu\cdot\bn - \chi)(\lambda - \rho) = 0 \quad  \text{on $\{y = 0\}$}, \\
	\bu\cdot\bn = \hat{\bu}\cdot\bn \quad \text{on $\{x = 0\}$}, \label{eq:bc-exact-un}\\
	\sigma_{nn} = \hat{\sigma}_{nn} \quad \text{on $\partial\Omega \setminus (\{y = 0\}\cup \{x=0\})$},\\
	\sigma_{nt} = \hat{\sigma}_{nt} \quad \text{on $\partial\Omega$},
\end{align}
\end{subequations}
where $\hat{\sigma} = \sigma(\hat{\bu}, \hat{p})$. Boundary conditions are set for the normal velocity along $\{x=0\}$ in \eqref{eq:bc-exact-un} to mimic the boundary conditions enforced at $\Gamma_d$ and make $\dim{R_V} = 1$. In this case, $R_V$ is a one-dimensional vector space containing vertical motions. Therefore, 
\[
	R_V \cap K = \left\lbrace (0,\theta): \theta \leq \min{\chi} \right\rbrace,
\]
and, since $ F = A_\varepsilon\hat{\bu} + G_\varepsilon\hat{\bu} - B \hat{p} - \gamma_n'\hat{\lambda}$, we have that
\[
	\langle F, \br \rangle_V = \theta \int_0^1 x^\gamma\,\dd x < 0
\]
for all $\br = (0, \theta)$ with $\theta < 0$ (note that $\min{\chi} < 0$). This proves that the compatibility condition \eqref{eq:compatibility-condition} holds and the system is well-posed. Although the functional setting of this numerical test differs slightly from the setting studied in this paper, the numerical test contains the fundamental elements of the setting analysed.

We compute solutions to the $r$-Stokes system on $\Omega = (0,1)^2$ with boundary conditions \eqref{eq:bc-exact} on a sequence of uniformly refined meshes using the finite element spaces in \eqref{eq:fem-spaces}. The regularisation parameter in \eqref{eq:A-discrete} is set to $\varepsilon = 10^{-4}$. In Glen's law \eqref{eq:glens_law} we fix $\mathcal{A} = 0.5$ and for the friction boundary condition we set $\tau = 1$. We consider the values $n = 1$, 2, 3 and 4, which correspond with $r = 2$, 1.5, 1.33 and 1.25. The resulting orders of convergence for the velocity are shown in Tables \ref{table:orders} and \ref{table:orders_vLr}, and for the pressure and Lagrange multiplier in Table \ref{table:orders_pl}. For the Lagrange multiplier error, we use the discrete norm
\[
	\nrm{\mu_h}_{\Sigma',h} = h^{1/r'} \nrm{\mu_h}_{L^{r'}(\Gamma_b)},
\]
which should yield the same order of convergence as the one that would be obtained with the $\Sigma'$ norm by a standard inverse inequality.

\begin{table}
\centering
\caption{Calculated orders of convergence for the velocity computed with a manufactured solution together with estimated orders according to Theorem \ref{thm:convergence}.}
\label{table:orders}
	\begin{tabular}{c|cccc|cccc}
	\toprule
		& \multicolumn{4}{|c|}{$\nrm{\bD(\bu) - \bD(\bu_h)}_{L^r(\Omega}$} & \multicolumn{4}{c}{$\nrm{\bu - \bu_h}_V$} \\
		\cmidrule{2-9}
		$h \backslash r$ &2.00& 1.50& 1.33& 1.25 & 2.00& 1.50& 1.33& 1.25 \\
		\midrule
		$3.54 \times 10^{-1}$& - & - &- & - &- & - &- & - \\
		$1.77 \times 10^{-1}$& 0.96& 1.05& 1.08& 1.11& 0.97& 1.10& 1.14& 1.18\\
		$8.84 \times 10^{-2}$& 0.97& 1.03& 1.05& 1.07& 0.98& 1.06& 1.09& 1.12\\
		$4.42 \times 10^{-2}$& 0.97& 1.02& 1.04& 1.06& 0.98& 1.04& 1.06& 1.08\\
		$2.21 \times 10^{-2}$& 0.97& 1.02& 1.03& 1.04& 0.98& 1.03& 1.04& 1.06\\
		$1.10 \times 10^{-2}$& 0.97& 1.01& 1.02& 1.03& 0.98& 1.02& 1.03& 1.04\\
		\midrule
		1 & 1.00& 1.00 & 1.00 & 1.00 & 1.00& 1.00 & 1.00 & 1.00 \\ 
	\bottomrule
	\end{tabular}	
\end{table}

\begin{table}
\centering
\caption{Calculated orders of convergence for the velocity in the $L^r(\Omega)$-norm computed with a manufactured solution.}
\label{table:orders_vLr}
	\begin{tabular}{c|cccc}
	\toprule
		& \multicolumn{4}{|c|}{$\nrm{\bu - \bu_h}_{L^r(\Omega}$} \\
		\cmidrule{2-5}
		$h \backslash r$ &2.00& 1.50& 1.33& 1.25 \\
		\midrule
		$3.54 \times 10^{-1}$ &- & - &- & - \\
		$1.77 \times 10^{-1}$ & 1.97& 2.07& 1.99& 1.85\\
		$8.84 \times 10^{-2}$ & 1.95& 1.95& 1.87& 1.72\\
		$4.42 \times 10^{-2}$ & 1.96& 1.96& 1.86& 1.71\\
		$2.21 \times 10^{-2}$ & 1.96& 1.97& 1.87& 1.72\\
		$1.10 \times 10^{-2}$ & 1.96& 1.98& 1.87& 1.74\\
	\bottomrule
	\end{tabular}	
\end{table}

\begin{table}
\centering
\caption{Calculated orders of convergence for the pressure and the Lagrange multiplier computed with a manufactured solution together with estimated orders according to Theorem \ref{thm:convergence}.}
\label{table:orders_pl}
	\begin{tabular}{c|cccc|cccc}
	\toprule
		& \multicolumn{4}{|c|}{$\nrm{p-p_h}_Q$} & \multicolumn{4}{c}{$\nrm{\lambda - \lambda_h}_{\Sigma',h}$} \\
		\cmidrule{2-9}
		$h \backslash r$ &2.00& 1.50& 1.33& 1.25 & 2.00& 1.50& 1.33& 1.25 \\
		\midrule
		$3.54 \times 10^{-1}$& - & - &- & - &- & - &- & - \\
		$1.77 \times 10^{-1}$& 0.88& 0.93& 0.96& 0.98& 1.00& 1.00& 1.00& 0.98\\
		$8.84 \times 10^{-2}$& 0.90& 0.94& 0.97& 0.98& 1.00& 1.00& 0.98& 0.93\\
		$4.42 \times 10^{-2}$& 0.91& 0.95& 0.97& 0.95& 1.01& 1.00& 0.96& 0.87\\
		$2.21 \times 10^{-2}$& 0.92& 0.95& 0.97& 0.90& 1.01& 1.00& 0.93& 0.80\\
		$1.10 \times 10^{-2}$& 0.93& 0.96& 0.96& 0.84& 1.01& 1.00& 0.88& 0.73\\
		\midrule
		$2/r' $& 1.00& 0.67& 0.5& 0.4& 1.00& 0.67& 0.5& 0.4 \\
	\bottomrule
	\end{tabular}	
\end{table}

Table \ref{table:orders} indicates that the orders of convergence for the velocity in the seminorm $\nrm{\bD(\cdot)}_{L^r(\Omega)}$ and in the $V$-norm coincide. This demonstrates that the presence of rigid modes in the velocity space does not affect the accuracy of the velocity computation in the $V$-norm. The computed orders of convergence for the velocity in the $V$-norm coincide with those estimated in \eqref{eq:error-estimate-V}. In Table \ref{table:orders_vLr} we see that the orders of convergence for the velocity in the $L^r(\Omega)$-norm appear to increase by one when compared to the orders computed with the $V$-norm. On the other hand, the orders of convergence obtained for the pressure appear to be independent of $r$. A closely related problem (without contact boundary conditions) is solved in the work of Belenki et al.~\cite{belenki2012}. In the work of Belenki et al., the problem is formulated as an $r$-Stokes problem with Dirichlet boundary conditions and the MINI element is used for the velocity and pressure. Interestingly, their numerical results deliver the predicted orders of convergence for the pressure error. Hence, the apparent suboptimality of \eqref{eq:error-estimate-QS} for the pressure could be due to the finite elements used here or to the presence of contact boundary conditions and a Lagrange multiplier. Regarding the Lagrange multiplier, the estimated orders of convergence are exceeded, but a dependence on $r$ is observed.

\section{Conclusions}

In this study, we present a Stokes variational inequality that arises when modelling a symmetrical marine ice sheet. We prove the well-posedness of this system whenever the subspace of rigid modes in the velocity space is of dimension at most one under the condition that a compatibility condition holds in Theorem \ref{thm:continuous-well-posed}. We consider a family of finite element discretisations for this problem in Section \ref{sec:abstract-discretisation} and prove an analogous well-posedness result for the discrete system in Theorem \ref{thm:discrete-well-posed}. Using the techniques from \cite{belenki2012, hirn2013}, we then establish error estimates for a particular finite element discretisation in Theorem \ref{thm:convergence}. These error estimates, which are verified with a numerical test using a manufactured solution, indicate that the presence of rigid modes and the nonlinearity of the friction boundary condition do not affect the order of convergence of the scheme.

The results from this study give a theoretical justification for using a finite element discretisation belonging to the family considered in Section \ref{sec:abstract-discretisation} in glaciological applications. Moreover, this analysis can be extended to different contact problems in glaciology such as the subglacial cavitation problem, where the subspace of rigid modes present in the velocity space is of dimension one whenever Dirichlet boundary conditions are enforced for the tangential component of the velocity on the top boundary, as done in \cite{gagliardini2007}. This extension would require taking into consideration the presence of non-homogeneous boundary conditions for the velocity and the use of a periodic domain. However, if instead of Dirichlet boundary conditions, we enforce Neumann boundary conditions on the top boundary, as described in \cite{dediego2022}, the space of rigid modes is then of dimension two. This situation would require a more complicated extension that should be considered in future work. Subglacial cavitation is considered in \cite{dediego2022} and the Stokes variational inequality that arises is solved using the finite element discretisation from Section \ref{sec:fe-scheme}.

A major assumption of this paper is that the domain is two-dimensional. An extension of the analysis presented here to three dimensions would require a careful consideration of the rigid modes present in the velocity space, since the space of rigid modes in three dimensions is larger than in two dimensions. However, in most problems of interest, three dimensional marine ice sheets are considered to be enclosed within two lateral walls, see for example \cite{favier2012}. In this case, if the lateral walls and the bedrock are flat, the space of rigid modes in $V$ is once again reduced to vertical movements and is therefore one-dimensional. As a consequence, much of the analysis from this paper would still be valid in three dimensions. However, the extension operator presented in Appendix \ref{app-subsec:extension}, used to prove Lemma~\ref{lemma:inf-sup-VSh}, relies heavily on the fact that the domain is two-dimensional. Therefore, the choice of finite elements used to solve the variational inequality would have to be chosen and studied carefully.

		
\appendix

\section{Equivalence of formulations}\label{app:equivalence-forms}

In this appendix we demonstrate the equivalence between the strong formulation \eqref{eq:main} of the contact problem with boundary conditions \eqref{eq:bc-gammas}-\eqref{eq:bc-gammad}, the variational inequality \eqref{eq:vi-continuous}, the minimisation of $\Jc$, defined in \eqref{eq:Jc-functional}, and the mixed formulation \eqref{eq:mixed-continuous}. This analysis is similar to the one presented in \cite{chen2013}, with the difference that in this case we consider contact boundary conditions.

\begin{lemma}
	If $(\bu,p)\in \bCc^2(\Omega)\times \Cc^1(\Omega)$, then the strong formulation \eqref{eq:main} with boundary conditions \eqref{eq:bc-gammas}-\eqref{eq:bc-gammad} holds if and only if the variational inequality \eqref{eq:vi-continuous} is satisfied.
\end{lemma}

\begin{proof}
	Let $(\bu,p)\in \bCc^2(\Omega)\times \Cc^1(\Omega)$ satisfy \eqref{eq:main} and \eqref{eq:bc-gammas}-\eqref{eq:bc-gammad}. It is clear that if \eqref{eq:main-subeq2} holds, then $\langle Bq,\bu\rangle_V = 0$ for all $q\in Q$. Let $\bv\in K$ and multiply \eqref{eq:main-subeq1} by $\bv-\bu$ and integrate over $\Omega$. The equality 
	\begin{equation}\label{eq:int-by-parts}
		\begin{split}
			- \int_\Omega \left[ \nabla \cdot \left( \alpha |\bD\bu|^{r-2} \bD\bu \right) - \nabla p \right] \cdot (\bv - \bu)\,\dd x = \\
			\langle A\bu - Bp, \bv - \bu \rangle_V - \int_{\partial\Omega} \sigma (\bv - \bu) \cdot \bn\,\dd s
		\end{split}
	\end{equation}
	follows from the divergence theorem. We also have that 
	\[
		\int_{\partial\Omega} \sigma (\bv - \bu) \cdot \bn\,\dd s = \int_{\partial\Omega} \left( \sigma_{nn} (\bv - \bu) \cdot \bn + \bsigma_{nt} \cdot (\bv - \bu )\right)\,\dd s.
	\]
	As a result of the contact conditions \eqref{eq:main-bc-contact}, we have 
	\begin{align*}
		\int_{\Gamma_b} \sigma_{nn} (\bv - \bu)\cdot\bn\,\dd s \geq - \int_{\Gamma_b} p_w \left( \bv - \bu \right) \cdot \bn\,\dd s,
	\end{align*}
	from which the variational inequality \eqref{eq:vi-continuous} follows.
	
	The converse statement is deduced by means of the integration by parts formula \eqref{eq:int-by-parts} and the use of the fundamental lemma of calculus of variations with adequate test functions. The examples in \cite{glowinski1981,haslinger1996} contain similar derivations.
\end{proof}

For the next step, we need to use the inf-sup condition \eqref{eq:inf-sup-VbQ} between the velocity and the pressure spaces. 

\begin{lemma}\label{lemma:equivalence-vi-min}
	Given a solution $(\bu,p)\in K\times Q$ of the variational inequality \eqref{eq:vi-continuous}, the velocity field is then divergence free, i.e.~$\bu\in \Ko$, and is a minimiser of the functional $\Jc:\Ko\to \RR$ defined in \eqref{eq:Jc-functional}. Conversely, if $\bu\in \Ko$ minimises $\Jc:\Ko\to \RR$, then there is a unique $p\in Q$ such that $(\bu,p)\in K\times Q$ solves \eqref{eq:vi-continuous}.
\end{lemma}

\begin{proof}
	For the first part of the Lemma, for a test function $\bv\in \Ko$, the variational inequality \eqref{eq:vi-continuous} can be written as 
	\begin{align}\label{eq:vi-continuous-divfree}
		\langle A\bu + G\bu - F, \bv - \bu \rangle_V \geq 0 \quad \forall \bv\in \Ko.
	\end{align}
	By the convexity of $\Jc$ and the fact that $\langle D\Jc(\bu), \bv \rangle_V = \langle A\bu + G\bu - F,\bv\rangle_V$, it follows from \eqref{eq:vi-continuous-divfree} that $\Jc(\bu) \leq \Jc(\bv)$ for all $\bv\in \Ko$.
	
	Conversely, if we assume $\bu\in\Ko$ to minimise $\Jc$ over $\Ko$, then $\bu$ solves \eqref{eq:vi-continuous-divfree}. Now, using \cite[Lemma 3.3]{amrouche1994}, we can decompose $\bv\in K$ into the sum $\bv = \bv_0 + \bw$ of a divergence-free velocity field $\bv_0\in \Ko$ and the field $\bw\in V_b$. Then, the variational inequality \eqref{eq:vi-continuous} will hold if there is a $p\in Q$ such that
	\begin{align}\label{eq:ve-Vb}
		\langle A\bu - F, \bw \rangle_V = \langle Bp,\bw \rangle_V \quad \forall \bw\in V_b.
	\end{align}
	By \eqref{eq:inf-sup-VbQ}, there is a $p\in Q$ for which \eqref{eq:ve-Vb} holds and it is unique.
\end{proof}

Finally, we show that the variational inequality \eqref{eq:vi-continuous} is equivalent to the mixed problem \eqref{eq:mixed-continuous}. This proof relies on the fact that the range of the operator $\gamma_n : V \to \Sigma$ is closed.

\begin{lemma}\label{lemma:equivalence-mixed-vi}
	If $(\bu,p)\in K\times Q$ solves the variational inequality \eqref{eq:vi-continuous}, then there is a unique $\lambda\in\Lambda$ such that $(\bu,p,\lambda)$ is a solution of the mixed problem \eqref{eq:mixed-continuous}. Conversely, if $(\bu,p,\lambda)\in V\times Q\times \Lambda$ solves \eqref{eq:mixed-continuous}, then $(\bu,p)$ is a solution of \eqref{eq:vi-continuous}.
\end{lemma}

\begin{proof}
	Equation \eqref{eq:mixed-continuous-V} can be rewritten as 
	\begin{align}\label{eq:mixed-continuous-rewritten}
		\gamma_n'\lambda = A\bu + G\bu - Bp - F \quad \text{in $V'$}.
	\end{align}
	Since $\gamma_n: V \to \Sigma$ has a closed range, we have that $\Ran{\gamma_n'} = \left( \Ker{\gamma_n}\right)^\circ$, where
	\[
		\left( \Ker{\gamma_n}\right)^\circ = \left\lbrace \mu\in\Sigma' : \langle \mu, \phi \rangle_\Sigma = 0 \quad \forall \phi\in\Ker{\gamma_n}\right\rbrace.
	\]
	Therefore, if $(\bu,p)\in K\times Q$ is a solution to \eqref{eq:vi-continuous}, then there is unique $\lambda\in \Sigma'$ if $A\bu + G\bu - Bp - F\in \left( \Ker{\gamma_n}\right)^\circ$. For a $\bw\in\Ker{\gamma_n}$, we clearly have that $\bu + \bw \in K$. Using the variational inequality \eqref{eq:vi-continuous}, we can write
	\[
		\langle A\bu + G\bu - Bp - F, \bw \rangle_V = 0,
	\]
	which means that $A\bu + G\bu - Bp - F\in \left( \Ker{\gamma_n}\right)^\circ$. Next, we must show that $\lambda\in \Lambda$ and that \eqref{eq:mixed-continuous-S} holds. By setting $\bv = 0$ and $\bv = 2\bu$ in \eqref{eq:vi-continuous} we see that $\langle \lambda, \gamma_n\bu\rangle_\Sigma = 0$. Since $\gamma_n\bu \leq 0$ in $\Sigma$, it follows that \eqref{eq:mixed-continuous-S} must hold. Finally, $\lambda\in\Lambda$ follows from \eqref{eq:vi-continuous}, \eqref{eq:mixed-continuous-rewritten} and the fact that $\bv + \bu \in K$ for any $\bv \in K$.
	
	For the second part of the lemma, if $(\bu,p,\lambda)\in V\times Q\times \Lambda$ solves \eqref{eq:mixed-continuous}, then $\nabla\cdot \bu = 0$ a.e. in $\Omega$, $\langle\lambda, \gamma_n\bv\rangle_\Sigma \geq 0$ for all $\bv \in K$ and $\langle\lambda, \gamma_n\bu\rangle_\Sigma = 0$. This implies that $\langle\mu, \gamma_n\bu\rangle_\Sigma \geq 0$ for all $\mu\in \Lambda$, hence $\bu\in \Ko$. The variational inequality then follows directly from \eqref{eq:mixed-continuous} by testing with $(\bv - \bu, q)$, where $(\bv,q)\in K\times Q$.	
\end{proof}

\section{Technical results on finite element spaces}

\subsection{Approximation properties in negative order Sobolev spaces} \label{app-subsec:interpolation}

Several technicalities arise from the need to handle the dual space of the fractional Sobolev space $W^{1-1/r,r}(\Gamma_b)$ and its finite element approximation $\Sigma_h$. Let $s\in(0,1)$ and $m\in [1,\infty)$; on $\Gamma_b$, the norm of the fractional Sobolev space $W^{s,m}(\Gamma_b)$ with $s\in(0,1)$ and $m\in [1,\infty]$ can be defined by
\begin{align}\label{eq:fractional_norm}
	\nrm{\phi}^m_{W^{s,m}(\Gamma_b)} = \nrm{\phi}^m_{L^m(\Gamma_b)} + [\phi]^m_{W^{s,m}(\Gamma_b)},
\end{align}
where
\[
	[\phi]^m_{W^{s,m}(\Gamma_b)} = \int_{\Gamma_b}\int_{\Gamma_b} \frac{|\phi(x) - \phi(y)|^m}{|x-y|^{1+sm}}\,\dd x\dd y,
\]
see \cite{nezza2012}. In order to prove certain approximation properties on $\Sigma_h$ we need to introduce some theoretical results. We start by defining the following pair of spaces
\[
	L_0^m(e) = \left\lbrace \phi\in L^m(e) : \int_{e} \phi\,\dd x = 0 \right\rbrace,\quad W^{s,m}_0(e) = W^{s,m}(e)\cap L^r_0(e)
\]
for an edge $e\in\Ec(\Gamma_b)$. We can use the fractional normed Poincar\'e inequality proved in \cite[Lemma 7.1]{ern2017} to show that
\begin{align}\label{eq:poincare-ineq}
	\nrm{\phi}_{L^m(e)} \leq |e|^s [\phi]_{W^{s,m}(e)} \quad \forall \phi\in W_0^{s,m}(e).
\end{align}
Inequality \eqref{eq:poincare-ineq} can be extended to negative norms by writing the $L^m(e)$ norm for any $\phi\in L^m(\Gamma_b)$ as
\[
	\nrm{\phi}_{L^m(e)} = \sup_{\psi\in L^{m'}(e)}{\frac{\int_{e}\phi\psi\,\dd s}{\nrm{\psi}_{L^{m'}(e)}}}
\]
and deducing that, if $\phi\in L^m_0(e)$, we have 
\begin{equation}\label{eq:negative-norm}
	\begin{split}
		\nrm{\phi}_{\left( W^{s,m}(e) \right)'} &= \sup_{\psi\in W_0^{s,m}(e)}{\frac{\int_{e}\phi\psi\,\dd s}{\nrm{\psi}_{W^{s,m}(e)}}} \\
		&\leq |e|^s \sup_{\psi\in L^{m'}(e)}{\frac{\int_{e}\phi\psi\,\dd s}{\nrm{\psi}_{L^{m'}(e)}}} = |e|^s  \nrm{\phi}_{L^{m'}(e)}.
	\end{split}
\end{equation}

For the finite element space $\Sigma_h$ defined in \eqref{eq:Sh}, let $\pi_\Sigma:L^m(\Gamma_b) \to \Sigma_h$ be the standard interpolation operator onto piecewise constant polynomials which takes the average of functions over each $e\in \Ec(\Gamma_b)$. Then, from inequality \eqref{eq:poincare-ineq} we can prove error estimates in fractional norms. Moreover,  \eqref{eq:negative-norm} leads to
\begin{align}
	\nrm{\phi - \pi_\Sigma\phi}_{\left( W^{s,m}(\Gamma_b) \right)'} \lesssim h^{2s} \nrm{\phi}_{W^{s,m'}(\Gamma_b)},
\end{align}
where $h = \max{ \{|e| : e \in \Ec(\Gamma_b)\}}$.

\subsection{An interpolation operator for the velocity}\label{app-subsec:interpolation-V}

Here, we compile a variety of results from different sources and prove a result regarding an interpolation operator for the velocity that preserves the discrete divergence and maps elements of $K$ into $K_h$. We denote by $\pi_V$ the interpolation operator introduced in \cite[Section 3.1]{girault2003} that is defined as follows for each component of a vector-valued function: for a non-degenerate simplex $K\in\Tc_h$ with edges $\{e_i\}_{i = 1}^3$ and vertices $\{\ba_i\}_{i=1}^3$, we define the nodal basis functions $\phi_{x}$ with $x\in \{e_i\}_{i = 1}^3 \cup \{\ba_i\}_{i=1}^3$ by
\begin{align*}
	\phi_{\ba_i}(\ba_j) &= \delta_{ij}, \quad \int_{e_j} \phi_{\ba_i}\,\dd s = 0,\quad \int_{e_j} \phi_{e_i}\,\dd s = \delta_{ij}, \quad \phi_{e_i}(\ba_j) = 0,
\end{align*}
for all $i,j\in\{1,2,3\}$. For each vertex $\ba_i$ we choose a an edge $e_{\ba_i}\in \{e_i\}_{i = 1}^3$ such that $\ba_i\in \overline{e_{\ba_i}}$. We then define the dual basis functions $\{\psi_{\ba_i}\}_{i=1}^3$ by
\begin{align*}
	\int_{e_{\ba_i}}\psi_{\ba_i} \phi_{x}\,\dd s = \delta_{\ba_ix}, \quad \psi_{\ba_i}\in\bbP_2(e_{\ba_i}),
\end{align*}
where $x$ denotes the edge $e_{\ba_i}$ or its two end-points. We then define (the scalar version of) $\pi_V$ by 
\begin{align*}
	(\pi_V u)(\bx) = \sum_{i = 1}^3 \left( \left[ \int_{e_{\ba_i}} u \psi_{\ba_i}\,\dd s\right] \phi_{\ba_i}(\bx) + \left[ \int_{e_{i}} u \,\dd s\right] \phi_{e_i}(\bx) \right).
\end{align*}

When considering the definition of $\pi_V$ in terms of a triangulation $\Tc_h$ of $\Omega$, for vertices $\ba \in \partial \Omega$, we set the associated edge $e_{\ba}$ to also be contained in $\partial \Omega$. Then, we have that $\pi_V(V)\subset V_h$ and we can prove \eqref{eq:F-tang-approx}.

If the spaces $V_h$, $Q_h$ and $\Sigma_h$ are defined as in \eqref{eq:fem-spaces}, a straightforward consequence of the definition of $\pi_V$ is that
\begin{align}
	\langle Bq_h, \bv \rangle_V &= \langle Bq_h, \pi_V \bv \rangle_V && \forall (\bv,q_h)\in V\times Q_h,\label{eq:div-preserving}\\
	\langle \mu_h, \gamma_n\bv \rangle_\Sigma &= \langle \mu_h, \gamma_n \pi_V \bv \rangle_\Sigma && \forall (\bv,\mu_h)\in V\times \Sigma_h. \label{eq:K-preserving}
\end{align}
Additionally, the interpolation operator $\pi_V$ has two key approximation properties. First, the optimal approximation property
\begin{align}\label{eq:optimal-approx}
	\nrm{\bv - \pi_V\bv}_{\bW^{s,m}(\Omega)} \lesssim h^k \nrm{\bv}_{\bW^{s+k,m}(\Omega)}
\end{align}
holds for all $m\geq 0$ and $s,k\in\NN$ such that $0 \leq s \leq 3$ and $0\leq k \leq 3-s$. Finally, given the operator $\bF$ defined in \eqref{eq:def-F}, the additional approximation property holds:
\begin{align}\label{eq:F-approx}
	\nrm{\bF(\bD\bv) - \bF(\bD\pi_V\bv)}_{L^2(\Omega)} \lesssim h \nrm{\nabla\bF(\bD\bv)}_{L^2(\Omega)}.
\end{align}
Property \eqref{eq:optimal-approx} is shown to hold in \cite{girault2003}. On the other hand, \eqref{eq:F-approx} follows from \cite[Theorem 3.4]{belenki2012} by applying Poincar\'e's inequality once points (a) and (b) from Assumption 2.9 in that reference are shown to hold. These two points result from \eqref{eq:div-preserving} and \eqref{eq:optimal-approx}.

Finally, we may also prove that 
\begin{align}\label{eq:F-tang-approx}
	\nrm{\bF(\bT\bv) - \bF(\bT\pi_V\bv)}_{L^2(\Gamma_b)} \lesssim h \sum_{\substack{e\in\Ec(\partial\Omega)\\ \overline{e}\cap\overline{\Gamma}_b\neq\emptyset}} \nrm{\nabla\bF(\bT\bv)}_{L^2(e)}
\end{align}
by imitating the proof for \cite[Theorem 3.4]{belenki2012} and applying Poincar\'e's inequality. Most of the steps in this proof draw from algebraic relations for the function $\bF$ and the N-functions considered therein that continue to be valid in our context. Additionally, we need the following Orlicz-continuity result analogous to that of \cite[Theorem 3.2]{belenki2012}: for an N-function $\psi$ with $\Delta_2(\psi) < \infty$ and an edge $e\in\Ec(\Gamma_b)$, 
\begin{align*}
	\int_e \psi(|\bT\pi_V\bv|)\,\dd s \lesssim \sum_{\substack{e'\in\Ec(\partial\Omega)\\ \overline{e}\cap\overline{e'}\neq\emptyset}} \int_{e'} \psi(|\bT\bv|)\,\dd s.
\end{align*}
We may show the above inequality to hold by following the proof of \cite[Theorem 4.5]{diening2007} and using the local $L^1$-estimate for $e\in\Ec(\Gamma_b)$:
\begin{align*}
	\int_e |\bT\pi_V\bv|\,\dd s \lesssim \sum_{\substack{e'\in\Ec(\partial\Omega)\\ \overline{e}\cap\overline{e'}\neq\emptyset}} \int_{e'} |\bT\bv|\,\dd s.
\end{align*}
To prove this inequality, we turn to the definition of $\pi_V$ and use the bounds 
\begin{align*}
	\nrm{\psi_{\ba}}_{L^\infty(e_{\ba})}\lesssim |e|^{-1}, \quad \nrm{\phi_{\ba}}_{L^1(e_{\ba})}\lesssim |e| \quad \text{and} \quad \nrm{\phi_{e}}_{L^1(e)}\lesssim 1.
\end{align*}

\subsection{An extension operator}\label{app-subsec:extension}

In this section we prove an auxiliary result required for the proof of Lemma \ref{lemma:inf-sup-VSh}. We build an extension operator $\Phi:\Sigma\to V_h$ which is uniformly bounded and satisfies
\begin{align}\label{eq:appC2-property}
	\langle \mu_h, \gamma_n(\Phi\phi) \rangle_\Sigma = \langle \mu_h, \phi \rangle_\Sigma \quad \forall \mu_h \in \Sigma_h.
\end{align}

\textit{Step 1}. We will first find a uniformly bounded linear operator $\Pi:\Sigma\to \gamma_n(V_h)$ with the property that 
\begin{align}\label{eq:appC2-step1}
	\int_e (\phi - \Pi\phi)\,\dd s = 0 \quad \text{for any $e\in\Ec(\Gamma_b)$ and $\phi\in \Sigma$.}
\end{align}
Let 
\[
	Z_h = \left\lbrace \phi_h\in \Cc(\Gamma_b) : \phi_h|_e \in \Pc_2(e) \quad \forall e\in\Ec(\Gamma_b) \right\rbrace
\]
and note that $Z_h \subset \gamma_n(V_h)$. For $\phi\in\Sigma$ and $e\in\Ec(\Gamma_b)$, we define $\Pi_2:\Sigma\to Z_h$ by setting  
\begin{align*}
	(\Pi_2\phi)(a) &= 0  \quad \text{for the endpoints $a$ in $e$},\\
	\int_e \Pi_2\phi\,\dd s &= \int_e\phi\,\dd s.
\end{align*}
We clearly have that $\Pi_2\phi = 0 $ if and only if $\sum_{e\in\Ec(\Gamma_b} \int_e |\phi|\,\dd s = 0$, so the latter defines a norm on $\Pi_2(\Sigma)$. By exploiting this fact and the norm equivalence on finite dimensional spaces, one can see that 
\[
	\nrm{\Pi_2\phi}_{W^{1-1/r,r}(e)}\lesssim |e|^{-1/r'} \nrm{\phi}_{L^r(e)} \quad \forall e\in\Ec(\Gamma_b)
\]
for all $\phi\in\Sigma$. Now, let $\pi_{Z}:\Sigma\to Z_h$ be the quasi-interpolation operator defined in \cite{ern2017}. This operator is uniformly bounded in the $W^{1-1/r,r}(\Gamma_b)$ norm and satisfies 
\[
	\nrm{\phi - \pi_Z\phi}_{L^r(e)}\lesssim |e|^{1-1/r} \nrm{\phi}_{W^{1-1/r,r}(e)}
\]
for any edge $e\in \Ec(\Gamma_b)$ and function $\phi\in \Sigma$. As a result, the operator $\Pi = \pi_Z + \Pi_2(I-\pi_Z)$ is uniformly bounded and possesses the required property \eqref{eq:appC2-step1}.

\textit{Step 2.} For the final step, we define a uniformly bounded operator $\gamma_{n,h}^{-1}:\gamma_n(V_h) \to V_h$ for which $\gamma_{n,h}^{-1}\phi_h\cdot\bn = \phi_h$ on $\Gamma_b$. This operator can be defined as the solution of the problem:
\begin{align*}
	\int_\Omega \nabla (\gamma_{n,h}^{-1}\phi_h) : \nabla\bv_h\,\dd x &= 0 && \forall \bv_h\in V_h,\\
	\gamma_{n,h}^{-1}\phi_h\cdot\bn &= \phi_h && \text{on $\Gamma_b$}, \\
	\gamma_{n,h}^{-1}\phi_h\cdot\bn &= 0 && \text{on $\Gamma_d$}.
\end{align*}
Then, the operator $\Phi = \gamma_{n,h}^{-1} \circ \Pi$ is uniformly bounded and property \eqref{eq:appC2-property} holds.

\begin{remark}
	The construction of the uniformly bounded operator $\Pi:\Sigma\to \gamma_n(V_h)$ in step 1 above closely resembles that of the Fortin operator in \cite[Proposition 8.4.3]{boffi2013}. In fact, the operator $\Phi \circ \gamma_n: V \to V_h$ effectively acts as a Fortin operator in the proof of Lemma \ref{lemma:inf-sup-VSh}.
\end{remark}

\bibliographystyle{amsplain}

\bibliography{bibliography}

\end{document}